\documentclass[11pt]{amsart}
\usepackage[utf8]{inputenc}
\usepackage{amsmath,amsthm,amsfonts,amssymb,amscd,mathtools,mathdots}
\usepackage{verbatim}
\usepackage{mathrsfs}
\usepackage{multicol}
\usepackage{multirow} 
\usepackage{tabularx}
\usepackage[usenames,dvipsnames]{color}
\usepackage{enumerate}
\usepackage{graphicx}
\usepackage{color,xcolor}
\usepackage{bbm}
\usepackage[all]{xy}
\usepackage{hyperref}
\usepackage[capitalise]{cleveref}
\usepackage{pb-diagram,tikz-cd}

\theoremstyle{definition}
\newtheorem{thm}{Theorem}[subsection]

\newtheorem{cor}[thm]{Corollary}

\newtheorem{lem}[thm]{Lemma}

\newtheorem{ex}[thm]{Example}

\numberwithin{equation}{subsection}

\def\lie#1{\mathfrak{#1}}

\def\hlie#1{\hat{\mathfrak{#1}}}
\def\bs#1{\boldsymbol{#1}}
\def\res{{\rm res}}
\def\ch{{\rm ch}}

\def\chgr{{\rm ch}_{\rm gr}}
\def\endd{\hfill$\diamond$}

\def\soc{{\rm soc}}

\def\tab{{\rm Tab}}
\def\rab{{\rm Rab}}
\def\wt{{\rm wt}}
\def\qbinom#1#2{\begin{bmatrix} #1\\#2\end{bmatrix}}
\def\tqbinom#1#2{\text{$\left[\begin{smallmatrix} #1\\#2\end{smallmatrix}\right]$}}

\begin{document}

\title[Tensor Products, Excellent Filtrations, Weyl Group Orbits, and Tableaux Conting]{Tensor Product Decompositions, \\ Limits in Excellent Filtrations,\\ Affine Weyl Group Orbits,\\ and Tableaux Counting}
\author[L. Estivalez, A. Moura]{Laura Estivalez and Adriano Moura}
\thanks{The work of A.M. was partially supported by Fapesp grant 2018/23690-6. Part of this work was performed during L.E. M.Sc. studies and another part during her Ph.D. studies which were supported by the Coordenação de Aperfeiçoamento de Pessoal de Nível Superior – Brasil (CAPES) – Finance Code 001 and Fapesp grant 
2021/13007-0. We thank Deniz Kus and Valentin Rappel for useful discussions concerning socle computations and Dyck paths.}

\address{Departamento de Matemática, Universidade Estadual de Campinas, Campinas - SP - Brazil, 13083-859.}
\email{aamoura@unicamp.br, l265199@dac.unicamp.br}
	
\begin{abstract}
	We express the outer multiplicities in the tensor products of two fundamental simple modules for an affine Kac-Moody algebra of type $A$ in terms of counting certain sets of multipartitions by exploring the stabilizing limits of certain excellent filtrations. This extends for all ranks a previously obtained result by Jakeli\'c and the second author for rank $1$. The same outer multiplicities were previously computed by Misra and Wilson in terms of counting certain sets of tableaux. By comparing these two expressions and by explicitly exhibiting a combinatorial description of level-$2$ affine Weyl group orbits, we establish the existence of a bijection between the Misra-Wilson set of tableaux and a disjoint union of certain sets of multipartitions. 
\end{abstract}
	
\maketitle

\section{Introduction}

There are many examples of ``counting identities'' which can be proved by representation theoretic arguments. Such proofs arise from establishing different combinatorial expressions for the same representation theoretic invariant. For instance, the invariant we shall focus on here is the multiplicity of a simple module $S$ as a summand of a given integrable module $V$ lying in category $\mathcal O$ for an affine Kac-Moody algebra $\hlie g$. Let
\begin{equation*}
	[V:S]\in\mathbb Z_{\ge 0}
\end{equation*}
denote this multiplicity. More precisely, we focus on the case that $V$ is a tensor product of two integrable simple modules from $\mathcal O$. As commented in the introduction of \cite{jamo:exf}, to which the present paper can be regarded as a follow up, there are several papers dedicated to the problem of computing such multiplicities. 

One of the most prominent combinatorial gadget in Lie representation theoretic realm is the concept of Young tableaux. For instance, a certain set of Young tableaux was used to provide a model for fundamental highest-weight crystals of type $A_n^{(1)}$ in \cite{jmmo,mm}. Let $\Lambda_i, i\in\hat I=\{0,1,\dots,n\}$, denote the fundamental weights of $\hlie g$ and let $\hat P^+$ be the corresponding set of dominant affine weights. For $\xi\in\hat P^+$, we let $V(\xi)$ denote a simple module of highest weight $\xi$. Using these crystal models, it was proved in \cite{mw:tpgen} that
\begin{equation}\tag{1}\label{e:mwint}
	[V(\Lambda_0)\otimes V(\Lambda_i):V(\xi)] = \tau^n_i(\eta)
\end{equation} 
where $\eta = \Lambda_0+\Lambda_i-\xi$ and $\tau^n_i(\eta)$ counts the elements of a certain set of extended Young tableaux which we refer to as the set of MW $i$-tableaux. The precise definition of such tableaux is reviewed (or rather, rephrased) in \cref{ss:mwt}. 

On the other hand, the main result of \cite{jamo:exf} gives an expression for computing
\begin{equation*}
	[V(\Lambda_0)\otimes V(\Lambda):V(\xi)]
\end{equation*}
for any $\Lambda,\xi\in\hat P^+$ in the case that the underlying simple Lie algebra $\lie g$ is simply laced, in terms of multiplicities in excellent filtrations, also known as Demazure flags. 
The expression has the following form
\begin{equation}\tag{2}\label{e:jmi}
	[V(\Lambda_0)\otimes V(\Lambda):V(\xi)] =  \sum_{ \mu\in P^-_\xi}  \max_{\gamma\in P^-_\Lambda}\ [D(\gamma): D(\mu)].
\end{equation}
Here, $P^-_\xi$ denotes a certain subset of the orbit of $\xi$ by the action of the affine Weyl group and $D(\mu)$ is the Demazure module generated by the weight space $\mu$ inside $V(\xi)$. Also, if $\ell$ is the level of $\xi$, so the level of $\Lambda$ is $\ell-1$, $[D(\gamma): D(\mu)]$ denotes the multiplicity of $D(\mu)$ in a level-$\ell$ Demazure flag for $D(\gamma)$. 
If $\lie g$ is of type $A_1$, the $\max$ in \eqref{e:jmi} was realized in \cite{jamo:exf} as a stabilizing limit. We extend this result for type $A_n$ in \cref{t:typeAn}, whose statement has the form
\begin{equation}\tag{3}\label{e:jmlimi}
	[V(\Lambda_0)\otimes V(\Lambda):V(\xi)] =  \sum_{\mu\in P^-_\xi}  \lim_{k\to\infty} [D(\gamma_k): D(\mu)],
\end{equation}
where $\gamma_k$ is an explicitly (and easily) described sequence of elements in $P^-_\Lambda$. 

In the case that $\ell=2$, the numbers $[D(\gamma): D(\mu)]$ have been computed in \cite{bcsw:mdpl2} in terms of products of $q$-binomials (the precise statement is recalled in \eqref{e:multflagAn}). Such expression establishes a connection between Demazure-flag multiplicities and partitions (see \eqref{e:binmpart}). Once this connection is established, \cref{l:mpart=0} essentially explains how to compute the limit in \eqref{e:jmlimi}. This leads to the main representation theoretic result of the present paper, \cref{t:typeAnpart}, whose statement has the form
\begin{equation}\tag{4}\label{e:ompi}
	[V(\Lambda_0)\otimes V(\Lambda_i):V(\xi)] = \sum_{\mu\in P^-_\xi} \rho_{\bs \mu_0}(f_{i,\xi}(\mu)).
\end{equation}  
Here, $f_{i,\xi}$ is a certain quadratic function explicitly defined in \eqref{e:formwt} and $\rho_{\bs b}(m)$, where $\bs b=(b_1,\dots,b_n)\in\mathbb Z^n_{\ge 0}$, counts the elements of a certain set of multipartitions of $m\in\mathbb Z$ (see \eqref{e:multpartc}), whose parts are bounded by the numbers $b_j, 1\le j\le n$. Comparing this with \eqref{e:mwint}, we get
\begin{equation}\tag{5}\label{e:combidi}
	\tau^n_i(\eta) = \sum_{\mu\in P^-_\xi} \rho_{\bs \mu_0}(f_{i,\xi}(\mu)).
\end{equation}
Our main combinatorial result, \cref{t:maincomb}, arises from this after a combinatorial description of the set $P^-_\xi$ is obtained by characterizing the affine Weyl group orbit of $\xi$ in \cref{ss:orbits}. Thus, the ``tableaux counting'' given by $\tau^n_i$ can be done by a ``partition counting'' without any reference to representation theory of Lie algebras. This is illustrated in \cref{ex:counting}. 

The aforementioned results solve a part of the problems mentioned in \cite{jamo:exf} as topics for future publications. Let us comment on those which are not treated here.  While \eqref{e:jmi} is valid for all simply laced $\lie g$ and $\Lambda$ of all levels, \eqref{e:jmlimi} has been proved only for type $A$. The core of the proof is \cref{c:inccof}, which allows us to produce the sequence $\gamma_k$. We have a candidate for a corresponding result for type $D$ and the proof is apparently within reach. Extending for type $E$ is also probably not out of reach, so analogues of \eqref{e:jmlimi} for these types should not be too difficult to obtain. We have also obtained formulas similar to \eqref{et:socle'} for types $B,C,D$, and $G$ and, hence, purely combinatorial descriptions of the sets $P^-_\xi$  similar to \eqref{e:Gxipart'} for theses types are also within reach and we expect they should appear in a forthcoming publication soon motivated by other kind of applications. 

The real difficulty apparently lies in computing the limit on the right-hand side of \eqref{e:jmlimi} for types $D$ and $E$ and even for type $A$ in the case $\ell>2$. The approach we used here rely on formulas for Demazure-flag multiplicities in terms of $q$-binomials. Such formula exists for type $A_1$ and $\ell=3$ \cite[Proposition 1.4]{bcsv:dfcpmt}. In this case, following step-by-step the computations we have done here should lead to an expression similar to \eqref{e:ompi} for tensor products of the form $V(\Lambda_0)\otimes V(\Lambda)$ with $\Lambda$ of level $2$. We are not aware of formulas of this type for higher $\ell$ or other types beside $A_1$. Still in type $A_1$, it was obtained in \cite{bk} an expression for multiplicities in level-$\ell$ Demazure flags for level-$1$ Demazure modules for any $\ell$ in terms of certain sets of Dyck paths. For $\ell=2$, we have established a bijection of such Dyck paths with certain sets of partitions and recovered the results of \cite{jamo:exf}. For general $\ell$, the expressions from \cite{bk} can be used to produced alternating expressions for level-$\ell$ Demazure flags for level-$(\ell-1)$ Demazure modules in terms of Dyck paths, so they can be plugged in on the right-hand-side of \eqref{e:jmlimi}. However, so far, we have not found a way of computing the limit using such expressions  if $\ell>2$. Thus, finding a way of computing the limit in \eqref{e:jmlimi} beyond the aforementioned cases remains a ground for investigations. Of course, searching for formulas with  flavor similar to \eqref{e:mwint} for more general tensor products may also be an interesting topic of investigation, leading to purely combinatorial statements in the spirit of \cref{t:maincomb}. 

The paper is organized as follows. \cref{s:parttab} is dedicated to setting up the notation related to partitions and tableaux, culminating with the statement of \cref{t:maincomb} and the aforementioned \cref{ex:counting}. We start \cref{s:tpd} by setting the Lie theoretic notation for finite and affine Kac-Moody algebras and their representations. The needed background on Demazure modules and flags are reviewed in \cref{ss:dem}. \cref{s:tpd} culminates with the statements of the main results on tensor product multiplicities: \cref{t:typeAn}, \cref{t:typeAnpart}, and \eqref{e:tpgenA}. \cref{s:orbits} is dedicated to the characterization of affine Weyl group orbits, leading to a combinatorial description of the sets $P^-_\xi$ in \eqref{e:Gxipart'}. We illustrate this description in special cases in \cref{ex:sup2soc}. A subcase of \cref{ex:sup2soc}, combined with \cref{t:typeAnpart} and \eqref{e:mwint}, leads to a proof of \cref{t:maincomb} in \cref{ss:maincomb}. The proof of \cref{t:typeAn} is given in the first two subsections of \cref{s:omp}, while the proof of \cref{t:typeAnpart} is given in the other two subsections.

\section{Partitions and Tableaux}\label{s:parttab}

Throughout the paper, let $\mathbb C, \mathbb R, \mathbb Q$, and $\mathbb Z$ denote the sets of complex, real, rational, and integer numbers, respectively. Let also $\mathbb Z_{\ge m} ,\mathbb Z_{< m}$, etc.,  denote the obvious subsets of $\mathbb Z$. Given $a,b\in\mathbb Z$, we let $[a,b]$ be the corresponding interval $	[a,b] = \mathbb Z_{\ge a}\cap \mathbb Z_{\le b}$. 
We denote by $\Sigma_n$ the symmetric group seen as permutation of $[1,n]$. The cardinality of a set $A$ will be denoted by $|A|$.
The symbol $\cong$ means ``isomorphic to''. We shall use the symbol $\diamond$ to mark the end of remarks, examples, and statements of results whose proofs are postponed. The symbol \qedsymbol\ will mark the end of proofs as well as of statements whose proofs are omitted.

\subsection{Partitions}
Given $\bs m=(m_1,\dots,m_l) \in \mathbb Z^l$ and $k\in\mathbb Z$, set 
\begin{equation}\label{e:|part|}
	|\bs m|_j = \sum_{i=1}^j m_i \ \ \text{for}\ \ 0\le j\le l, \quad |\bs m|=|\bs m|_l, 
\end{equation}
and
\begin{equation}\label{e:sumntopart}
	k\pm\bs m = (k\pm m_1,\dots,k\pm m_l).
\end{equation}  
Let
\begin{equation*}
	\mathcal S^l(m) = \{\bs m\in \mathbb Z_{\ge 0}^l:|\bs m|=m\}.
\end{equation*}
By a partition of $m\in\mathbb Z_{\ge 0}$ of length at most $l$ we mean an element $\bs m\in\mathcal S^l(m)$ such that 
\begin{equation*}
	m_1\ge m_2\ge \cdots\ge m_l
\end{equation*}
The numbers $m_i$ with $m_i\ne 0$ will be referred to as the parts of $\bs m$. We denote the set of all such elements of $\mathbb Z^l$ by $\mathcal P^l(m)$ and identify $\mathcal P^{l-1}(m)$ with the subset of $\mathcal P^l(m)$ consisting of elements $\bs m$ such that $m_l=0$. Under this identification, one easily sees
\begin{equation}\label{e:stabnparts}
	l\ge m \quad\Rightarrow\quad \mathcal P^l(m) = \mathcal P^k(m) \ \ \text{for all}\ \ k\ge l.
\end{equation}
We also set
\begin{equation*}
	\mathcal P(m) = \bigcup_{l\ge 0}\mathcal P^l(m) = \mathcal P^m(m), \quad \mathcal P^l =\bigcup_{m\ge 0}\mathcal P^l(m), \quad \mathcal P = \bigcup_{m\ge 0}\mathcal P(m),
\end{equation*} 
and refer to the elements of $\mathcal P(m)$ as the partitions of $m$ and to elements of $\mathcal P$ as partitions. We also let
\begin{equation}
	\mathcal P^{=l}(m) = \mathcal P^l(m)\setminus\mathcal P^{l-1}(m) 
\end{equation}
denote the set of partitions of $m$ with exactly $l$ parts, and similarly define $\mathcal P^{=l}$. If $\bs m\in\mathcal P^{=l}$ we say $\bs m$ has length $\ell(\bs m)=l$. 

Given $\bs m\in\mathcal P^l(m)$, let
\begin{equation*}
	\mathscr P(\bs m) = \{m_j: 1\le j\le l, m_j\ne 0\}
\end{equation*}
be the set consisting of the parts of $\bs m$. Suppose 
\begin{equation*}
	\mathscr P(\bs m) = \{k_1,k_2,\dots,k_s\} \ \ \text{with}\ \ k_1>k_2> \cdots > k_s
\end{equation*}
and let $r_j$ be the multiplicity of $k_j$ as a part of $\bs m$. In that case, we shall also use the following notation:
\begin{equation}\label{e:partmulnot}
	\bs m = (k_1^{r_1}, k_2^{r_2}, \cdots, k_s^{r_s}).
\end{equation}

 Consider the natural action of the symmetric group $\Sigma_l$ on $\mathbb Z^l$, i.e.,
\begin{equation*}
	\sigma(m_1,\dots,m_l) = (m_{\sigma(1)},\dots,m_{\sigma(l)})
\end{equation*}
and, given $\bs m\in\mathbb Z^l_{\ge 0}$, denote by $\overline{\bs m}$ the unique element of the orbit $\Sigma_l\bs m$ which is a partition of $|\bs m|$. So, by definition 
\begin{equation}\label{e:barm}
	\{\overline{\bs m}\} = \mathcal P^l\cap \Sigma_l\bs m.
\end{equation}

Given $b\in\mathbb Z_{\ge 0}$, let $\mathcal P_b(m)$ be the subset of $\mathcal P(m)$ of partitions whose parts are bounded by $b$, $\mathcal{P}^{l}_{b}(m) = \mathcal P^l(m)\cap\mathcal P_b(m)$, and  similarly define $\mathcal P^{=l}_b(m)$, and so on. Set also $\rho(m)=|\mathcal P(m)|$, and similarly define $\rho^l(m), \rho_b(m)$, and $\rho^l_b(m)$. We set
\begin{equation*}
	\rho(m)=0 \ \ \text{for}\ \ m\notin\mathbb Z_{\ge 0}
\end{equation*}
and similarly in the case of  $\rho^l(m)$, etc.. For $\bs m \in \mathcal P^l_b(m)$, define 
\begin{equation}\label{e:addlp}
	(b,\bs m) = (b,m_1 , \cdots , m_l)\in \mathcal P^{l+1}_b(m+b)
\end{equation}
and 
\begin{equation}\label{e:difp}
	\bs m^- = (m_1-m_2,m_2-m_3,\dots,m_{l-1}-m_l,m_l)\in\mathbb Z_{\ge 0}^l.
\end{equation}

Finally, we will also need a certain set of multipartitions defined as follows. Given $\bs a,\bs b\in\mathbb Z^l_{\ge 0}$, set
\begin{equation}\label{e:multpart}
		\mathcal{P}_{\bs b}(m)  = \bigcup_{\bs m\in\mathcal S^l(m)} \mathcal{P}_{b_1}(m_1)\times \cdots \times \mathcal{P}_{b_l}(m_l), \quad 
		\mathcal{P}_{\bs b}^{\bs a}(m)  = \bigcup_{\bs m\in\mathcal S^l(m)} \mathcal{P}_{b_1}^{a_1}(m_1)\times \cdots \times \mathcal{P}_{b_l}^{a_l}(m_l),
\end{equation}
$\rho_{\bs b}(m)=|\mathcal{P}_{\bs b}(m)|$, and $\rho^{\bs a}_{\bs b}(m)=|\mathcal{P}^{\bs a}_{\bs b}(m)|$.  Thus,
\begin{equation}\label{e:multpartc}
	\rho_{\bs b}(m) = \sum_{\bs m\in\mathcal S^l(m)} \prod_{j=1}^l \rho_{b_j}(m_j) \quad\text{and}\quad \rho_{\bs b}^{\bs a}(m) = \sum_{\bs m\in\mathcal S^l(m)} \prod_{j=1}^l \rho_{b_j}^{a_j}(m_j).
\end{equation}
Note
\begin{equation}
	b_j = 0 \quad\Rightarrow\quad \rho_{\bs b}(m) = \rho_{\bs b'}(m)
\end{equation}
where $\bs b'\in\mathbb Z^{l-1}$ is obtained from $\bs b$ by deleting the $j$-th entry.  Also,
\begin{equation}
	\bs b\in\mathbb Z^l_{>0} \quad\Rightarrow\quad \rho_{\bs b}(1) = l.
\end{equation}

\subsection{A Tableaux Counting Problem}\label{ss:mwt}

We identify a partition $\bs m$ with the Young diagram whose $r$-th row has $m_r$ boxes. We denote this diagram by $D_{\bs m}$ and we say it has shape $\bs m$,  as usual. By the box in position $(r,c)$ we mean the $c$-th box in the $r$-th row of the diagram. In other words, we identify $D_{\bs m}$ with the following subset of $\mathbb Z^2$:
\begin{equation*}
	D_{\bs m} = \{(r,c)\in\mathbb Z^2: 1\le c\le m_r, 1\le r\le \ell(\bs m)\}. 
\end{equation*}
The partition with no parts is associated to the empty diagram $D_\emptyset$. 

Fix $n\in\mathbb Z_{>0}$, let $\hat I=[0,n]$, and, given $a,b\in\mathbb Z$,  write
\begin{equation}\label{e:congn}
	a\equiv b \ \ \text{if}\ \  a-b\in(n+1)\mathbb Z.
\end{equation}
By an extended tableau of rank $n$ and shape $\bs m$ we mean a function
\begin{equation*}
	T:D_{\bs m}\to \hat I.
\end{equation*} 
As usual, this can be described in terms of boxes by saying that $T$ is a Young diagram of shape $\bs m$ whose boxes have been filled in with elements of $\hat I$. If a given box has been filled in with $i\in\hat I$, we shall say that box has content $i$.  Henceforth, we refer to an extended tableau of rank $n$ simply by a tableau. 
We denote by $\tab$ the set of all  tableaux and by $\tab_{\bs m}$ the subset of tableaux of shape $\bs m\in\mathcal P$. We also consider a tableau associated to the partition having no parts, which is denoted by $T_\emptyset$. 

Consider the function
\begin{equation}
	\gamma:\tab\to \mathbb Z^{n+1}, \quad T\mapsto \gamma(T) = (\gamma_0(T),\dots, \gamma_n(T))
\end{equation}
where $\gamma_i(T), i\in\hat I$, is the number of boxes in $T$ whose content is $i$. We think of $\gamma(T)$ as the ``content character'' of $T$. 
Note
\begin{equation*}
	T \in\tab_{\bs m} \quad\Rightarrow\quad \bs m\in\mathcal P(|\gamma(T)|).
\end{equation*}
Thus,
\begin{equation*}
	\gamma^{-1}(\bs\eta)\subseteq \bigcup_{\bs m\in\mathcal P(|\bs\eta|)} \tab_{\bs m} \quad\text{for all}\quad \bs\eta\in\mathbb Z_{\ge 0}^{n+1}. 
\end{equation*}
A natural combinatorial problem at this point is to obtain methods for counting the number of tableaux of a given character, i.e., for computing
\begin{equation*}
	|\gamma^{-1}(\bs\eta)| \ \ \text{for}\ \ \bs\eta\in\mathbb Z_{\ge 0}^{n+1}. 
\end{equation*}
There are certainly several purely combinatorial approaches for performing this counting in terms of sets of (multi)partitions. Our goal, however, is to explore the connection of such kind of problem with representation theory of affine Kac-Moody algebras, so we will not pursue this direction here. More precisely, we will present a formula, in terms of multipartitions and whose proof arises from representation theoretic arguments, for computing 
\begin{equation*}
	|\gamma^{-1}(\bs\eta)\cap \mathbb W_i| 
\end{equation*}
where $\mathbb W_i$ is the subset of $\tab$ consisting of regular $i$-charged tableaux satisfying a couple of extra conditions, for each $i\in\hat I$. Since these conditions originally arose in \cite[Lemma 2.3]{mw:tpgen}, we shall refer to such tableaux as Misra-Wilson tableaux of charge $i$ or, simply, MW $i$-tableaux.  The remainder of this section is dedicated to reviewing the definition of such tableaux. The set of regular $i$-charged tableaux was used as a model for the affine $i$-th fundamental highest-weight crystal of type $A_n^{(1)}$ in \cite{jmmo,mm}. We follow closely the review made in \cite{mw:tpgen}, with modified notation. 

A partition $\bs m$ is $n$-regular (or simply regular if $n$ is clear from the context) if each part of $\bs m$ repeats at most $n$ times. In the spirit of \eqref{e:partmulnot}, this can be written as
\begin{equation}
	\bs m = (k_1^{r_1},k_2^{r_2},\dots,k_s^{r_s}) \ \ \text{with}\ \ r_i\le n \ \ \text{for all}\ \ 1\le i\le s.
\end{equation}
A tableau with regular shape will be said a regular tableau. In such a tableau, there are at most $n$ rows having any given number of boxes. Let us denote by $\rab$ the set of all regular tableaux and set $\rab_{\bs m} = \rab\cap\tab_{\bs m}$. A nonempty tableau $T$ is said to have charge $i\in\hat I$ if
\begin{equation}
	T(r,c) \equiv c-r+i 
\end{equation}
for all possible values of $(r,c)$. The empty tableaux will be regarded as an $i$-charged tableau for all $i\in\hat I$. Otherwise, there exists exactly one $i$-charged tableau for each shape, which we denote by $T_{\bs m,i}$ (see \cite[Figure 1]{mw:tpgen}). Thus, we have injective  maps
\begin{equation}
	\kappa_i: \mathcal P \to \tab, \ \ \bs m\mapsto T_{\bs m,i} \quad\text{for}\quad i\in\hat I. 
\end{equation}
The image of $\kappa_i$, i.e., the set of all $i$-charged tableaux, will be denoted by $\tab^i$ and we also set $\rab^i =\tab^i\cap \rab$. Let $\mathcal M_i$ denote the subset of $\mathcal P$ satisfying the following  condition:
\begin{equation}\label{e:mwcond}
	\bs m = (k_1^{r_1},k_2^{r_2},\dots,k_s^{r_s}) \in\mathcal M_i \quad\Leftrightarrow\quad k_l+i\equiv r_l +2 \sum_{j=1}^{l-1}r_j \ \ \text{for all}\ \ 1\le l<s.
\end{equation}
The original phrasing of the above in \cite[Lemma 2.3]{mw:tpgen} was
\begin{equation*}
	\bs m = (k_1^{r_1},k_2^{r_2},\dots,k_s^{r_s}) \in\mathcal M_i \quad\Leftrightarrow\quad r_1\equiv i + k_1  \ \ \text{and}\ \ k_{l+1}-k_l\equiv r_l+r_{l+1} \ \ \text{for all}\ \ 1\le l<s,
\end{equation*}
which is easily seen to be equivalent to \eqref{e:mwcond}. 
Finally, the set of MW $i$-tableaux is
\begin{equation}
	\mathbb W_i:= \kappa_i(\mathcal M_i)\cap\rab^i. 
\end{equation}

\subsection{A Counting Matching}
Let $\hat I=[0,n]$ as in the previous section and denote by $\mathbb P_i^+, i\in\hat I$, the subset of $\mathbb Z^{n+1}_{\ge 0}$ consisting of the elements $\bs\eta$ satisfying 
\begin{equation}\label{e:domceta}
	\eta'_r:=\delta_{0,r}+\delta_{i,r}-2\eta_r+\eta_{r-1}+\eta_{r+1}\ge 0 \quad\text{for all}\quad r\in\hat I. 
\end{equation} 
Here, for $l\in\mathbb Z$, we have set $\eta_l=\eta_r$ if $l\equiv r\in\hat I$. 
Let $\bs e_j, j\in\hat I$, be the element $(0,\dots,0,1,0\dots,0)\in\mathbb R^{n+1}$ where $1$ is in the entry $j+1$. The proof of \cite[Lemma 2.4]{mw:tpgen} shows
\begin{equation}\label{e:xi=jk}
	\bs\eta\in\mathbb P_i^+ \quad\Rightarrow\quad \bs\eta' = \bs e_j+\bs e_k \ \ \text{for some}\ \ j,k\in\hat I \ \ \text{such that}\ \ j+k\equiv i.
\end{equation}

\begin{ex}\label{ex:domn=2}
	For $n=2$, it follows that 
	\begin{equation*}
		\bs\eta\in\mathbb P_1^+\quad\Leftrightarrow\quad \bs\eta = \eta_0(1,1,1) -\epsilon\,\bs e_2\ \ \text{with}\ \ \epsilon\in\{0,1\}.  
	\end{equation*}
	Indeed, if $\epsilon = 0$, $\bs\eta'=(1,1,0)$ and if $\epsilon=1$, $\bs\eta'=(0,0,2)$. Note the former corresponds to $j=0, k=1$ in \eqref{e:xi=jk}, while the latter corresponds to $j=k=2$. No other pair $(j,k)$ satisfies $j+k\equiv 1$, so these are the only cases.
	\endd
\end{ex}

Set
\begin{equation}
	\tau_i^n(\bs\eta) = |\gamma^{-1}(\bs\eta)\cap \mathbb W_i|.
\end{equation}
Let us prepare the notation for presenting a matching of the counting function $\tau^n_i(\bs\eta), \bs\eta\in\mathbb P^+_i$, with the counting of the union of certain sets of the form $\mathcal P_{\bs b}(m)$.  
Given $\bs m\in\mathbb Z^n$, recall \eqref{e:congn} and let $\res (\bs m)\in\hat I$ be the element satisfying
\begin{equation}\label{e:resp}
	\res(\bs m)+|\bs m|\equiv 0.
\end{equation}
Given $(\bs m,\bs p)\in[1,2]^n\times \mathbb Z^n$, set
\begin{equation}\label{e:acomb}
	\bs a(\bs m,\bs p) = (a_1,\dots,a_n) \quad\text{with}\quad a_i=2p_i+m_i,
\end{equation}
and let $\mathscr P^+$ be the subset of $[1,2]^n\times \mathbb Z^n$ whose elements satisfy
\begin{equation}\label{e:pfordomcomb}
	p_i\ge -1 \ \forall\ i\in I, \quad p_i=-1 \ \Rightarrow\ m_i=2, \quad\text{and}\quad  p_1\ge p_2\ge\cdots\ge p_n. 
\end{equation}
Equivalently, 
\begin{equation*}
	(\bs m,\bs p)\in\mathscr P^+ \quad\Leftrightarrow\quad \bs a(\bs m,\bs p)\in\mathcal P^n. 
\end{equation*}

Given $1\le s\le n+1$, set
\begin{equation}\label{e:ms2}
	\bs m(s) = (2^{s-1},1^{n+1-s})\in\mathcal P^n_{2}.
\end{equation}  
Let also $\mathcal O(s)$ be the orbit of $\bs m(s)$ in $\mathbb Z^n$ under the usual action of $\Sigma_n$. Given $j,k\in\hat I$, let 
\begin{equation}\label{e:parpar2}
	\mathcal S_{j,k} = \{s\in[1,n+1]: s\equiv \pm|j-k| \}.
\end{equation}
For $s\in\mathcal S_{j,k}$, let also $p_{j,k}(s)\in\hat I$ be defined by 
\begin{equation}\label{e:ps2}
	 p_{j,k}(s)+1\equiv 
	\begin{cases}
		\max\{j,k\}, & \text{if } s\equiv |j-k|, \\ \min\{j,k\}, & \text{if } s\equiv - |j-k|
	\end{cases}
\end{equation}
and set
\begin{equation}\label{e:Gjk}
	\Gamma_{j,k} = \{(\bs m,\bs p)\in\mathscr P^+: \bs m\in\mathcal O(s) \text{ and }\res(\bs p)=p_{j,k}(s) \text{ for some } s\in\mathcal S_{j,k}\}. 
\end{equation}
Note $\mathcal S_{k,j}=\mathcal S_{j,k},\Gamma_{k,j}=\Gamma_{j,k}$, and $p_{j,k}(s)=p_{k,j}(s)$. 

\begin{ex}\label{ex:Gn=2}
	Note $\mathcal S_{j,j} = \{n+1\}, p_{j,j}(n+1)\equiv j-1$ and $\bs m(n+1)=(2^n)$. In particular, $\mathcal O(\bs m(n+1))=\{(2^n)\}$.  Then,
	\begin{equation}
		((2^n),\bs p)\in\Gamma_{j,j} \quad\Leftrightarrow\quad p_1\ge \cdots\ge p_n\ge -1 \ \ \text{and}\ \  |\bs p|\equiv 1-j. 
	\end{equation}
	In that case,
	\begin{equation}\label{e:aj=k}
		\bs a((2^n),\bs p) = 2(1+\bs p).
	\end{equation}
	This gives a complete description of $\Gamma_{j,j}$ for all $n$ and $j\in\hat I$ as well as of the corresponding elements $\bs a(\bs m,\bs p)$. For $n=2$, the pair $j=k=2$ appeared as a relevant case in \cref{ex:domn=2}. Thus, in that case,
	\begin{equation}
		\Gamma_{2,2} = \{((2^2), (3k-1-p,p)): p\ge -1, 3k\ge 2p+1\}. 
	\end{equation}

	The other relevant pair in \cref{ex:domn=2} was $(j,k)=(0,1)$. Thus, let us describe $\Gamma_{0,1}$ for $n=2$. In this case,
	\begin{equation*}
		\mathcal S_{0,1} = \{1,2\}, \ \ p_{0,1}(1) = 0, \ \ p_{0,1}(2) = 2, \ \ \bs m(1) = (1^2), \ \ \bs m(2) = (2,1),
	\end{equation*}
	so 
	\begin{equation*}
		\mathcal O(1) = \{(1^2)\} \ \ \text{and}\ \ \mathcal O(2) = \{(2,1), (1,2)\}. 
	\end{equation*}
	Then,
	\begin{equation}\label{e:p012}
		\begin{aligned}
			& ((1^2),\bs p)\in\Gamma_{0,1} &&\quad\Leftrightarrow\quad \ \bs p = (3k-p,p) \ \ \text{with}\ \ p\ge 0, 3k\ge 2p, \\
			& ((2,1),\bs p)\in\Gamma_{0,1} &&\quad\Leftrightarrow\quad \ \bs p = (3k+1-p,p) \ \ \text{with}\ \ p\ge 0, 3k\ge 2p-1, \\
			& ((1,2),\bs p)\in\Gamma_{0,1} & &\quad\Leftrightarrow\quad \ \bs p = (3k+1-p,p) \ \ \text{with}\ \ p\ge -1, 3k\ge 2p-1.
		\end{aligned}
	\end{equation}
	The corresponding elements $\bs a(\bs m,\bs p)$ are, respectively, 
	\begin{equation}\label{e:a012}
		(2(3k-p)+1,2p+1), \quad (2(3k+2-p),2p+1), \quad (2(3k+1-p)+1,2(p+1)).
	\end{equation}
	\endd
\end{ex}

For $\bs a\in\mathbb R^n$, set 
\begin{equation}\label{e:fpart}
	f(\bs a)=\frac{1}{n+1}\left(n\sum_{i\in I}a_i^2 - 2\sum_{1\le i<j\le n} a_ia_j\right).
\end{equation}
One easily checks $f$ is a positive definite quadratic form. Let $\pi:\mathbb R^{n+1}\to\mathbb R^n$ be the linear map whose kernel is spanned by $\bs e_0$ and set
\begin{equation}
	\varpi_i =\sum_{j=1}^i \pi(\bs e_j)\in\mathbb R^n \quad\text{for}\quad i\in\hat I. 
\end{equation}
In particular, $\varpi_0=0$ and 
\begin{equation}\label{e:normfundi}
	f(\varpi_i) = \frac{i(n+1-i)}{n+1}  \quad\text{for all}\quad i\in\hat I.
\end{equation} 
We are ready to state the main result of this section. Recall \eqref{e:multpart}.

\begin{thm}\label{t:maincomb}
	Let $i\in\hat I$ and $\bs\eta\in\mathbb P_i^+$ and suppose $\bs\eta'=\bs e_j+\bs e_k$ with $j,k\in\hat I$. Then, 
	\begin{equation}\label{e:maincomb}
		\tau_i^n(\bs\eta) = \sum_{(\bs m,\bs p)\in\Gamma_{j,k}} \rho_{\bs b(\bs m,\bs p)} \left( f_{i,\bs\eta}(\bs a(\bs m,\bs p))\right),
	\end{equation}
	where 
	\begin{equation}\label{e:formcomb}
		f_{i,\bs\eta}(\bs a) = \frac{f(\varpi_i)}{2} - \frac{f(\varpi_j+\varpi_k)}{4} +\eta_0-\frac{f(\bs a)}{4}
	\end{equation}
	and $\bs b(\bs m,\bs p) = (b_1,\dots,b_n)$ with
	\begin{equation}\label{e:boundcomb}
		b_r = p_r-p_{r+1} + \frac{m_r-m_{r+1}-|m_r-m_{r+1}|}{2} \ \ \text{for}\ \ 1\le r\le n
	\end{equation}
	after setting $p_{n+1}:=-1$ and $m_{n+1}:=2$.\endd	
\end{thm}

Since $f$ is positive definite, there are only finitely many pairs $(\bs m,\bs p)$ in \eqref{e:maincomb} for which 
\begin{equation*}
	f_{i,\bs\eta}(\bs a(\bs m,\bs p))\in\mathbb Z_{\ge 0},
\end{equation*}
which implies
\begin{equation}\label{e:cbcomb}
	f(\bs a(\bs m,\bs p))\le 2f(\varpi_i) - f(\varpi_j+\varpi_k) +4\eta_0.
\end{equation}
Thus, the summation in \eqref{e:maincomb} is actually the finite sum over the pairs $(\bs m,\bs p)\in\Gamma_{j,k}$ satisfying \eqref{e:cbcomb}. The formula \eqref{t:maincomb} is a counting matching between the set of MW $i$-tableaux with content character $\bs\eta$ and the disjoint union of a certain family of sets of multipartitions indexed by $\Gamma_{j,k}$.  \Cref{t:maincomb} will be proved in \cref{ss:maincomb} as a corollary of the fact that both sides of \eqref{e:maincomb} are equal to the multiplicity of a certain simple module of the affine Ka-Moody algebra of type $A_n^{(1)}$ inside a tensor product of fundamental modules.  \Cref{t:maincomb} is also  a generalization of  \cite[Equation (2.6.4)]{jamo:exf} since it recovers that identity when $n=1$. 

\begin{ex}\label{ex:counting}
	Let us compute the right hand side of \eqref{e:maincomb} in the case $j=k$. In this case, it follows from \eqref{e:normfundi} that 
	\begin{equation*}
		\frac{f(\varpi_i)}{2} - \frac{f(2\varpi_j)}{4} = \frac{i(n+1-i)}{2(n+1)} - \frac{j(n+1-j)}{(n+1)},
	\end{equation*}
	while \eqref{e:xi=jk} implies $2j\equiv i$. We also have to compute $f$ on the elements $\bs a$ described in \eqref{e:aj=k}:
	\begin{equation}\label{e:faj=k}
		(n+1)f(\bs a) = 4n((3k-j+2-p)^2+(p+1)^2) - 8(3k-j+2-p)(p+1). 
	\end{equation}
	Now, let us specialize to the context of \cref{ex:domn=2}, so $n=2=j, i=1$, and $\bs\eta=(\eta,\eta,\eta-1), \eta\ge 1$. We get
	\begin{equation}
		\begin{aligned}
			f_{1,\bs\eta}(\bs a) & = \frac{1}{3} - \frac{2}{3} - \frac{2((3k-p)^2+(p+1)^2) - 2(3k-p)(p+1)}{3} +\eta\\
			& = \eta-1 - 2(3k^2-3kp-k+p^2+p).
		\end{aligned}
	\end{equation}
	Finally, the corresponding elements described in \eqref{e:boundcomb} are
	\begin{equation*}
		(3k-2p-1,p+1).
	\end{equation*}
	Hence,  \eqref{e:maincomb} becomes
	\begin{equation}\label{e:maincomb22}
	 \tau^2_1(\eta,\eta,\eta-1)= \sum_{p\ge -1}\sum_{k\ge \lceil\frac{2p+1}{3}\rceil}  \rho_{ (3k-2p-1,p+1)} \left( \eta-1 - 2(3k^2-3kp-k+p^2+p)\right). 
	\end{equation}
	For instance, if one chooses $\eta=6$, it is not difficult to check that the only pairs $(p,k)$ satisfying \eqref{e:cbcomb} are $(-1,0), (0,1)$, and $(1,1)$, so
	\begin{equation}
		\tau^2_1(6,6,5)= \rho_{(1,0)}(5) + \rho_{(2,1)}(1)+\rho_{(0,2)}(3) = 1+2+2=5. 
	\end{equation}
	In other words, for $n=2$, the number of MW $1$-tableaux with 6 boxes with content $0$, 6 boxes with content $1$, and 5 boxes with content $2$ is 5. One can check the 5 corresponding partitions are: 
	\begin{align*}
	    (15,2), \quad (12,5), \quad (9,8), \quad (9,3^2,1^2), \quad \text{and} \quad (6,5,4,1^2).
	\end{align*}  
	A similar formula for $\tau_1^2(\eta,\eta,\eta)$ can be obtained by working with the elements described in \eqref{e:p012} and \eqref{e:a012}.                          
\endd
\end{ex}

\section{Tensor Product Decompositions}\label{s:tpd}

\subsection{Finite Type Cartan Data}

Let $\lie g$ be a simply laced finite dimensional simple Lie algebra over $\mathbb C$ with a fixed triangular decomposition $\lie g=\lie n^-\oplus\lie h\oplus\lie n^+$. Denote by $R$, $R^+$, and $\Delta$ the sets of roots, positive roots and simple roots, respectively. We let $I$ denote the set of vertices of the Dynkin diagram of $\lie g$ and, given  $i\in I$,  let $\alpha_i$ and $\omega_i$, denote the corresponding simple root and fundamental weight, respectively. 
Fix a Chevalley basis $\{x_\alpha^\pm, h_i:\alpha\in R^+,i\in I\}$ and set $h_\alpha = [x_{\alpha}^+,x_{\alpha}^-]$. In particular,
\begin{equation*}
	h_i=[x_{\alpha_i}^+,x_{\alpha_i}^-] \quad\text{for all}\quad i\in I.
\end{equation*} 
Let $P$ and $Q$ denote the weight and root lattices of $\lie g$, respectively, while $P^+$ and $Q^+$ are the submonoids generated by the fundamental weights and the simple roots, respectively. The highest root of $\lie g$ will be denoted by $\theta$. We normalize the corresponding symmetric bilinear form $(\ ,\ )$ such that $(\alpha,\alpha)=2$ for $\alpha\in R$. 
Consider also the linear operator $s_\alpha$ on $\lie h^*$ defined by
\begin{align*}
	s_\alpha\lambda=\lambda - (\lambda, \alpha) \alpha = \lambda - \lambda (h_\alpha) \alpha, \quad \alpha\in R^+,  \lambda\in\lie h^*,
\end{align*}
known as the reflection associated to $\alpha$.
We often simplify notation and write $x_i^\pm=x_{\alpha_i}^\pm, s_i=s_{\alpha_i}, i\in I$.  
Let $\mathcal W$ be the Weyl group of $\lie g$, which is a Coxeter group generated by the simple reflections $s_i, i \in I$. In particular, $s_\alpha\in\mathcal W$ for all $\alpha\in R^+$. 
Recall $\mathcal W$ has a unique longest element, which we denote by $w_o$.

\subsection{Affine Cartan Data}\label{ss:acd}

Given a Lie algebra $\lie a$ over $\mathbb C$, denote by $\lie a[t]$ the associated current algebra, i.e., $\lie a[t]=\lie a\otimes \mathbb C[t]$ with bracket $[x\otimes f,y\otimes g] = [x,y]\otimes (fg)$ for $x,y\in\lie a$ and $f,g\in\mathbb C[t]$. Set also  $\lie a[t]_+= \lie a\otimes t\mathbb C[t]$. The affine Kac-Moody algebra $\hlie g$ associated to $\lie g$ is the vector space $\hlie g=\lie g\otimes\mathbb C[t,t^{-1}]\oplus \mathbb Cc\oplus \mathbb Cd$ with bracket given by
\begin{equation*}
	[x \otimes t^r, y \otimes t^s] = [x, y] \otimes t^{r+s} + r\ \delta_{r,-s}\ (x, y)\ c, \quad [c,\hlie g]=\{0\}, \quad\text{and}\quad [d,x\otimes t^r]=r\ x\otimes t^r
\end{equation*}
for any $x,y \in \lie g$ and $r,s \in \mathbb Z$. 
We identify $\lie g$ and $\lie g[t]$ with the obvious subalgebras of $\hlie g$. Set
\begin{equation*}
	\hlie h = \lie h \oplus \mathbb C c \oplus \mathbb C d, \quad \hlie n^{+} = \lie n^{+} \oplus \lie g [t]_{+}, \quad\text{and}\quad \hlie b = \hlie n^{+} \oplus \hlie h.
\end{equation*}
Recall $\lie g[t]\oplus\mathbb Cc\oplus\mathbb Cd$ is a parabolic subalgebra of $\hlie g$ containing $\hlie b$. Identify $\lie h^*$ with the subspace $\{\lambda\in\hlie h^*:\lambda(c)=\lambda(d)=0\}$ and let $\Lambda_0,\delta\in\hlie h^*$ be defined by
\begin{equation*}
	\Lambda_0(d)=0=\Lambda_0(\lie h), \quad \Lambda_0(c)=1,\quad \delta(c)=0=\delta(\lie h), \quad \delta(d)=1.
\end{equation*}
Also, set $\hat I=I\sqcup\{0\}, \alpha_0=\delta-\theta, h_0=c-h_\theta$, and, for $i\in I$, set
\begin{equation} \label{e:Lambdaidef}
	\Lambda_i = \omega_i +\omega_i(h_\theta)\Lambda_0.
\end{equation}
Then, $\Lambda_i(h_j)=\delta_{i,j}$ for all $i,j\in\hat I, \{h_i:i\in\hat I\}\cup\{d\}$ is a basis of $\hlie h$, $\hat\Delta=\{\alpha_i:i\in\hat I\}$ is the set of simple roots for $\hlie g$, and $\hat R^+ = R^+\cup\{\alpha+r\delta: \alpha\in R\cup\{0\}, r\in\mathbb Z_{>0}\}$ is the set of positive roots. The affine root lattice $\hat Q$ is the $\mathbb Z$-span of $\hat\Delta$, while $\hat Q^+$ is the corresponding submonoid. 
We consider the usual partial order on $\hlie h^*$ defined by
\begin{equation*}
	\mu\le\lambda \quad\text{if}\quad \lambda-\mu\in \hat Q^+.
\end{equation*}
We also set
\begin{equation}\label{e:moddelta}
	\lambda\equiv\mu \quad\text{if}\quad \lambda-\mu\in\mathbb C\delta. 
\end{equation}
The symmetric bilinear form on $\lie h^*$ can be extended to one on $\hlie h ^*$ given by 
\begin{gather}\label{e:aform}
	(\alpha_i, \Lambda_0) = (\alpha_i,\delta)=  (\Lambda_0,\Lambda_0) = (\delta,\delta)=0, \quad \text{for} \quad i \in I, \quad\text{and}\quad (\Lambda_0,\delta)=1.
\end{gather}

Set $\hat P = \{\lambda\in\hlie h^*: \lambda(h_i)\in\mathbb Z \text{ for all } i\in\hat I\}$, $\hat P^+ = \{\lambda\in\hat P: \lambda(h_i)\ge 0 \text{ for all } i\in\hat I\}$, and note we have a direct sum of $\mathbb Z$-modules
\begin{equation}\label{e:hatPdec}
	\hat P  = P\oplus \mathbb Z\Lambda_0\oplus \mathbb C\delta 
\end{equation}
Given $\lambda\in\hlie h^*$, the number $\lambda(c)$ is called the level of $\lambda$ and we shall refer to $\lambda(d)$ as the degree of $\lambda$. Denote by $\bar\lambda\in P$ the projection of $\lambda$ on the first summand of the decomposition  \eqref{e:hatPdec}. Thus,
\begin{equation}\label{e:awdec}
	\lambda = \bar\lambda + \lambda(c)\Lambda_0 + \lambda(d)\delta 
\end{equation}
is the decomposition of $\lambda$ according to  \eqref{e:hatPdec}. Since $P\subseteq \mathbb QQ$, it follows from \eqref{e:aform} that
\begin{equation}\label{e:normlambda}
	(\lambda,\lambda) = (\bar\lambda,\bar\lambda) + 2\lambda(c)\lambda(d), \quad \lambda(c) = (\lambda,\delta), \quad\text{and}\quad \lambda(d) = (\lambda,\Lambda_0).
\end{equation}

The affine simple reflections $s_i, i\in\hat I$, can be defined as the linear endomorphism on $\hlie h^*$ such that
\begin{equation*}
	s_i\lambda = \lambda-\lambda(h_i)\alpha_i = \lambda-(\lambda,\alpha_i)\alpha_i \quad\text{for all}\quad \lambda\in\hat P.
\end{equation*}
For $i\ne 0$, the restriction of $s_i$ to $\lie h$ coincides with the previously defined simple reflections. Let
$\widehat{\mathcal W}$ denote the  affine Weyl group, which is generated by the simple reflections $s_i,i\in\hat I$. In particular, $\mathcal W$ is naturally a subgroup of $\widehat{\mathcal W}$ and
\begin{equation*}
	(w\lambda,w\mu) = (\lambda,\mu) \quad\text{for all}\quad w \in \widehat{\mathcal W}, \lambda,\mu\in\hlie h^*.
\end{equation*}
Moreover, since $\alpha_i(c)=0$ for all $i\in\hat I$, it follows that
\begin{equation*}
	(w\lambda)(c) = \lambda(c) \quad\text{for all}\quad w\in\widehat{\mathcal W}, \lambda\in\hlie h^*. 
\end{equation*}
These facts, together with \eqref{e:normlambda}, imply
\begin{align}\label{e:degorbit}
	\mu \in \widehat{\mathcal W} \lambda \quad\Rightarrow\quad 2\lambda(c)(\mu(d)-\lambda(d))=(\bar\lambda,\bar\lambda)-(\bar\mu,\bar\mu).
\end{align}
Indeed, 
\begin{align*}
	(\bar\mu,\bar\mu) + 2\lambda(c)\mu(d) =  (\bar\mu,\bar\mu) + 2\mu(c)\mu(d) = (\mu,\mu) =(\lambda,\lambda) = (\bar\lambda,\bar\lambda) + 2\lambda(c)\lambda(d).
\end{align*}
It is also well-known that
\begin{equation}\label{e:domorbit}
	\xi\in \widehat{\mathcal W}\hat P^+\setminus\mathbb C\delta \quad\Leftrightarrow\quad \xi\in\hat P \ \ \text{and}\ \ \xi(c)\in\mathbb Z_{>0},
\end{equation}
and, furthermore, 
\begin{equation}\label{e:domorbitu}
	\#(\widehat{\mathcal W}\xi \cap\hat P^+)\le 1 \quad\text{for all}\quad \xi\in\hlie h^*.  
\end{equation}
We denote by $\soc(\xi)$ the unique element of $\widehat{\mathcal W}\xi \cap\hat P^+$ when it exists. It follows from \eqref{e:degorbit} that
\begin{equation}\label{e:socdeg}
	\soc(\xi)(d) = \frac{(\xi,\xi)- (\overline{\soc(\xi)},\overline{\soc(\xi)}) }{2\xi(c)} = \xi(d) + \frac{ (\bar\xi,\bar\xi)- (\overline{\soc(\xi)},\overline{\soc(\xi)}) }{2\xi(c)}.
\end{equation}
Note \eqref{e:domorbitu} implies
\begin{equation}\label{e:socequiv}
	w\xi \equiv \lambda\in\hat P^+ \quad\Rightarrow\quad \soc(\xi)\equiv \lambda.
\end{equation}
In that case $\overline{\soc(\xi)}=\bar\lambda$ and the degree of $\soc(\xi)$ is given by \eqref{e:socdeg}. Thus, the task of finding $\soc(\xi)$ reduces to that of finding $\overline{\soc(\xi)}$.

It will be convenient to work with the alternative realization of  $\widehat{\mathcal W}$ as the semidirect product $\tau(Q) \ltimes \mathcal{W}$, where  $\tau$ is the map defined on $\hlie h^*$ which associates to each $\alpha\in\hlie h^*$ the linear map $t_\alpha:\hlie h^* \to \hlie h^*$ given by
\begin{align*}
	t_\alpha(\lambda)= \lambda + \lambda(c)\alpha - \left((\lambda,\alpha)+\frac{1}{2}(\alpha,\alpha)\lambda(c)\right)\delta,
\end{align*}
or, equivalently, 
\begin{align}\label{e:areflex}
	t_\alpha (\Lambda_0) = \Lambda_0 + \alpha - \frac{1}{2} (\alpha,\alpha) \delta, \qquad 
	t_\alpha (\omega_i) = \omega_i - (\omega_i,\alpha)\delta, \quad t_\alpha (\delta) = \delta, \qquad \alpha\in \hlie h^*, i \in I.
\end{align} 
The isomorphism $\widehat{\mathcal W}\to \tau(Q^\vee) \ltimes \mathcal{W}$ is given by $s_i\mapsto t_0s_i$ for $i\in I$ and $s_0 \mapsto t_\theta s_\theta$.
We keep the notation $s_i$ for the image of $s_i$ under this isomorphism.

\subsection{Basic Representation Theory Background}
If $V$ is a $\lie g$-module and $\mu\in\lie h^*$, the corresponding weight space is
\begin{equation*}
	V_\mu = \{v\in V: hv = \mu(h)v \ \forall\ h\in\lie h\}.
\end{equation*}
We shall only consider finite-dimensional $\lie g$-modules. In particular, 
\begin{equation*}
	V=\bigoplus_{\mu\in P} V_\mu \quad\text{and}\quad \dim V_\mu\in\mathbb Z.
\end{equation*}
The set $\wt(V) = \{\mu\in P:V_\mu\ne  0\}$ is the set of weights of $V$.
The character of $V$, denoted by $\ch V$, is the function $\ch V: P\to \mathbb Z, \mu\mapsto \dim V_\mu$. It can be naturally identified with the element
\begin{equation*}
	\sum_{\mu\in P} \dim V_\mu\ e^\mu \in \mathbb Z[P],
\end{equation*}
where $e^\mu$ denotes the element corresponding to $\mu$ in the group ring $\mathbb Z[P]$. It turns out $w\,\ch V = \ch V$, i.e., $\dim V_\mu = \dim V_{w\mu}$, for all $w\in\mathcal W$. 
We denote by $V(\lambda),\lambda\in P^+$, a simple $\lie g$-module with highest $\lambda$. 

These notions can be defined for a $\hlie g$-module in the category of integral modules in category $\mathcal O$. In particular, given such $\hlie g$-module and $\mu\in\hat P^+$, we let $V_\mu$ denote the associated weight space,  we denote by $\ch V$ the associated character, the characters are $\widehat{\mathcal W}$-invariant, the isomorphism classes of simple modules are in bijection with $\hat P^+$, and we denote by $V(\lambda),\lambda\in\hat P^+$, a simple module with highest-weight $\lambda$. 

Given a $\hlie g$-module $V$ such that
\begin{equation*}
	V \cong \bigoplus_{\lambda\in\hat P^+} V(\lambda)^{\oplus m_\lambda} \quad\text{for some}\quad m_\lambda\in\mathbb Z_{\ge 0},
\end{equation*}
we set
\begin{equation*}
	[V:V(\lambda)] = m_\lambda.
\end{equation*}
The number $[V:V(\lambda)]$ is often called the outer multiplicity of $\lambda$ in $V$, as opposed to the ``inner'' multiplicity which is $\dim V_\lambda$.  

The current algebra $\lie g [t]$ has a $\mathbb Z_{\ge 0}$-grading in the obvious way, which induces a grading on $U(\lie g[t])$ in such a way that an element
\begin{align*}
	(x_1 \otimes t_{r_1}) \cdots (x_s\otimes t_{r_s}), \quad  x_j \in \lie g, r_j \in \mathbb Z_{\geq 0}, 1 \leq j \leq s,
\end{align*}
has degree $r_1 + \cdots  + r_s$. Denote by $U(\lie g[t])_k$ the homogeneous component of degree $k$ and recall that it is a $\lie g$–module for all $k \in \mathbb Z_{\ge 0}$. A finite-dimensional $\mathbb Z$-graded $\lie g[t]$-module is a $\mathbb Z$-graded vector space $V$ admitting a compatible graded action of $\lie g[t]$. That is,
\begin{align*}
	V = \bigoplus_{k \in \mathbb Z} V[k] \quad\text{and}\quad (x \otimes t^r) V[k] \subseteq V[k+r] \ \ \text{for all}\ \ x\in\lie g, r\in\mathbb Z_{\ge 0}, k\in\mathbb Z.
\end{align*}
In particular, each graded piece $V [k]$ is a $\lie g$-module. If $\dim V [k] <\infty $ for all $k \in \mathbb Z$, we define the graded character as a formal sum
\begin{align*}
	\chgr V = \sum_{k \in \mathbb Z} \ch V[k] q^k \in P[\mathbb Z][[q]]. 
\end{align*}
Given a $\mathbb Z$-graded space $V$, let $\tau_p V$ be the graded space whose $r$-th graded piece is  $V [r - p]$. So
\begin{equation*}
	\chgr\tau_pV = q^p\chgr V. 
\end{equation*}
Given a $\lie g$-module $V$, we shall regard it as a $\mathbb Z$-graded $\lie g[t]$-module by setting
\begin{equation}\label{e:ev0def}
	V[0]=V \quad\text{and}\quad \lie g[t]_+V=0. 
\end{equation}

\subsection{Demazure Modules and Flags}\label{ss:dem}
The $\widehat{\mathcal W}$-invariance of $\ch V(\Lambda)$ implies $\dim V(\Lambda)_{w\Lambda}=1$ for all $\Lambda\in\hat P^+, w\in \widehat{\mathcal W}$. The Demazure module $V_w(\Lambda)$ associated to $\Lambda\in\hat P^+$ and $w\in \widehat{\mathcal W}$ is the $\hlie b$-submodule of $V(\Lambda)$ generated by $V(\Lambda)_{w\Lambda}$. It is usually denoted by $V_w(\lambda)$. 
Alternatively, given $\xi\in\widehat{\mathcal W}\hat P^+$, letting $\Lambda$ be the unique element in $\widehat{\mathcal W}\xi\cap\hat P^+$ and $w\in\widehat{\mathcal W}$ be such that $w\Lambda=\xi$, we have
\begin{equation*}
	V_w(\Lambda) = U(\hlie b)V(\Lambda)_\xi. 
\end{equation*}
This motivates the following alternative notation:
\begin{equation}\label{e:Demalt}
	D(\xi): = V_w(\Lambda). 
\end{equation}
Since $\hlie h^*\subseteq\hlie b$, $D(\xi)$ has a decompositon as a sum of affine weight spaces. Recall the definition of $\soc(\xi)$ in the line preceding \eqref{e:socdeg}. It follows from the definitions that  
\begin{equation*}
	\soc(\xi) = \Lambda \quad\text{and}\quad D(\xi)_{\soc(\xi)}=V(\Lambda)_\Lambda.
\end{equation*}

Let $\preccurlyeq$ denote the Bruhat order in $\widehat{\mathcal W}$. It is well-known (see \cite[Lemmas 1.3.20 and 8.3.3]{kumar}) that $(\widehat{\mathcal W},\preceq)$ is a directed poset,
\begin{gather}\notag 
	w\preceq w' \quad\Rightarrow\quad V_w(\Lambda)\subseteq V_{w'}(\Lambda),\\ \label{e:BruhatDem}\text{and}\\ \notag
	V(\Lambda) = \bigcup_{w\in \widehat{\mathcal W}} V_w(\Lambda).
\end{gather}
Since every element of $P$ is $\mathcal W$-conjugate to an element of $-P^+$, one easily sees that, for every $w\in\widehat{\mathcal W}$, there exists $w'\in{\mathcal W}$ such that
\begin{equation*}
	\ell(w'w)=\ell(w)+\ell(w') \qquad\text{and}\qquad w'w\Lambda(h_i)\le 0 \quad\text{for all}\quad i\in I.
\end{equation*}
Thus, if we set
\begin{equation*}
	\widehat{\mathcal W}_\Lambda^- = \left\{w\in\widehat{\mathcal W}: \overline{w\Lambda}\in -P^+\right\},
\end{equation*}
we get
\begin{equation}\label{e:demunion}
	V(\Lambda) = \bigcup_{w\in {\widehat{\mathcal W}}_\Lambda^-} V_w(\Lambda).
\end{equation}
The Demazure modules of the form $V_w(\Lambda)$ with $w\in {\widehat{\mathcal W}}_\Lambda^-$ are usually called $\lie g$-stable Demazure module due to the following well-known lemma.

\begin{lem}\label{l:gstdem}
	Let $\Lambda\in\hat P^+$ and $w\in \widehat{\mathcal W}$. The following are equivalent:
	\begin{enumerate}[(i)]
		\item $V_w(\Lambda)$ is a $\lie g[t]$-submodule of $V(\Lambda)$.
		\item $V_w(\Lambda)$ is a $\lie g$-submodule of $V(\Lambda)$.
		\item $\lie n^-V_w(\Lambda)_{w\Lambda} = 0$.
		\item $w\Lambda(h_i)\le 0$ for all $i\in I$.
	\end{enumerate}
Moreover, in that case, $V_w(\Lambda)$ is a finite-dimensional $\mathbb Z$-graded $\lie g[t]$-module and $U(\lie n^+)V(\Lambda)_{w\Lambda}$ is isomorphic to $V(\lambda)$ as a $\lie g$-module, where $\lambda=-w_o(\overline{w\Lambda})$.\qed
\end{lem}

We will work only with $\lie g$-stable Demazure modules. If $D(\xi)$ is such a module, then $\xi = \ell\Lambda_0+w_o\lambda + r\delta$ for some $\ell\in\mathbb Z_{>0}, \lambda\in P^+, r\in\mathbb C$, and we use the following notation which emphasizes the data $\ell,\lambda,r$:
\begin{equation*}
	D(\ell,\lambda,r)= D(\ell\Lambda_0 +w_o\lambda + r\delta) \quad\text{and}\quad D(\ell,\lambda) = D(\ell,\lambda,0).
\end{equation*}
 It is also well-known that the socle of $D(\xi)$ when regarded as a $\lie g[t]$-module is the $\lie g$-module
\begin{equation*}
	U(\lie n^-)D(\xi)_{\soc(\xi)} = U(\lie n^-)V(\soc(\xi))_{\soc(\xi)},
\end{equation*}
which is isomorphic to $V\left(\overline{\soc(\xi)}\right)$ as a $\lie g$-module. This explains our choice of notation for $\soc(\xi)$. 

Recall that a $\lie g[t]$-module $V$ admits a ($\lie g$-stable) Demazure flag if there exist $l>0, \xi_j\in \hat P$ such that $D(\xi_j)$ is $\lie g$-stable for $j\in [1,l]$, and a sequence of inclusions
\begin{equation}\label{e:demflag}
	0 = V_0 \subset V_1 \subset \cdots \subset V_{l-1} \subset V_l = V \quad\text{with}\quad V_j/V_{j-1}\cong D(\xi_j) \ \forall\ 1\le j\le l.
\end{equation}
If $\xi_j(c)=\ell$ for some $\ell$ and all $j$, such a sequence is said to be a level-$\ell$ Demazure flag for $V$. We shall only consider flags of a fixed level. Let $\mathbb V$ be a level-$\ell$ Demazure flag for $V$ and, given a Demazure module $D$, define the multiplicity of $D$ in $\mathbb V$ by
\begin{equation*}
	[\mathbb V:D] = \#\{1\le j\le l : V_j/V_{j-1}\cong D\}.
\end{equation*}
As observed in \cite[Section 2.8]{cssw:demflag}, if $\mathbb V'$ is another level-$\ell$ Demazure flag for $V$, then $[\mathbb V':D]=[\mathbb V:D]$. Hence, by abuse of language, we shift the notation from $[\mathbb V:D]$ to $[V:D]_\ell$. Also following \cite{cssw:demflag}, we consider the generating polynomial
\begin{equation}\label{e:gsdfm}
	[V:D]_\ell(q) = \sum_{m\in\mathbb Z}\ [V:\tau_mD]_\ell\ q^m \ \in\ \mathbb Z[q,q^{-1}].
\end{equation}
When no confusion arises, we simplify notation and drop the subindex $\ell$ from the brackets above. 

\subsection{Outer Multiplicities from Demazure Flags}\label{ss:omdf}
Given $\Lambda\in\hat P^+$, consider
\begin{equation*}
	\widehat{\mathcal W}_\Lambda^- = \left\{w\in\widehat{\mathcal W}: \overline{w\Lambda}\in -P^+\right\} \quad\text{and}\quad  P^-_\Lambda = \{w\Lambda : w \in \widehat{\mathcal W}^-_{\Lambda} \}\subseteq \hat P.
\end{equation*}
Since $w\delta = \delta$ for all $w\in\widehat{\mathcal W}$,  
\begin{equation}\label{e:Gammatrans}
	P^-_{\Lambda+s\delta} = P^-_\Lambda + s\delta \quad\text{for all}\quad \Lambda\in \hat P^+, s\in\mathbb C.
\end{equation}
The main result of \cite{jamo:exf} says that, for all $\Lambda\in\hat P^+$ such that $\Lambda(c)=\ell$, we have  
	\begin{equation}\label{t:multrel}
		[V(\Lambda_0)\otimes V(\Lambda):V(\xi)] =  \sum_{ \eta\in P^-_\xi}  \max_{\gamma\in P^-_\Lambda}\ [D(\gamma): D(\eta)] \quad\text{for all}\quad  \xi\in\hat P^+, \xi(c)=\ell+1.
	\end{equation}
It will be convenient to rephrase the above expression as follows. Given $\Phi\in\hat P^+$, an element of $P^-_\Phi$ can be written as $\Phi(d)\Lambda_0+w_o\mu+r\delta$ for some $\mu\in P^+$ and $r\in\mathbb C$. According to \eqref{e:socdeg}, we have 
\begin{equation*}
	\Phi(d) = r + \frac{(w_o\mu,w_o\mu)-(\bar\Phi,\bar\Phi)}{2\Phi(c)}.
\end{equation*}
Thus, setting 
\begin{equation}\label{e:socdeg'}
	r(\mu,\Phi) =  \Phi(d) - \frac{(\mu,\mu)-(\bar\Phi,\bar\Phi)}{2\Phi(c)} \quad\text{for}\quad \mu\in P^+, \Phi\in\hat P^+,
\end{equation}
and
\begin{equation}
	\Gamma_\Phi = \left\{\mu\in P^+: \Phi(d)\Lambda_0+w_o\mu+r(\mu,\Phi)\delta\in\widehat{\mathcal W}\Phi\right\},
\end{equation}
\eqref{t:multrel} becomes
\begin{equation}\label{e:multrel}
	[V(\Lambda_0)\otimes V(\Lambda):V(\xi)] =  \sum_{\mu\in \Gamma_\xi}  \max_{\nu\in \Gamma_\Lambda}\ [D(\ell,\nu,r(\nu,\Lambda)): D(\ell+1,\mu,r(\mu,\xi))]
\end{equation}
for all $\xi\in\hat P^+, \xi(c)=\ell+1$.

The following was proved in \cite[Proposition 2.6.1]{jamo:exf} in the case $\lie g=\lie{sl}_2$ and we shall extend it for type $A$ in general. The proof is an application of \eqref{t:multrel} together with \cref{c:inccof} below, which is obtained by studying certain ``canonical'' reduced expressions for elements of $\widehat{\mathcal W}$. Henceforth, we use the notation
\begin{equation*}
	|\lambda| = \sum_{i\in I} \lambda(h_i) \quad\text{for}\quad\lambda\in P.
\end{equation*}

\begin{thm} \label{t:typeAn}
	Assume $\lie g$ is of type $A_n$, let $\Lambda\in\hat P^+$, and set $\lambda=\bar\Lambda, \ell=\Lambda(c), s=\Lambda(d)$. Then, for all $\xi\in\hat P^+$ with $\xi(c)=\ell+1$, we have
	\begin{align*}
		[V(\Lambda_0)\otimes V(\Lambda):V(\xi)] =  \sum_{\mu\in\Gamma_{\xi}}\ \lim_{k\to\infty}\  [D(\ell,\lambda + \ell k \theta ):D(\ell +1,\mu,r(\mu,\xi)-s+k(|\lambda|+\ell k)].
	\end{align*}	\endd
\end{thm}

In the case $\ell=1$, we can make the right-hand side of the expression in \cref{t:typeAn} more explicitly as follows. Given $\mu \in P^+$, write 
\begin{equation}\label{e:mu01}
	-w_o\mu = 2\mu_0 + \mu_1 \ \ \text{with}\ \ \mu_0,\mu_1 \in P^+,\ \mu_1(h_i) \leq 1 \ \forall\ i\in I,
\end{equation}
and consider 
\begin{equation}\label{e:bsmuj}
	\bs\mu_0 = (\mu_0(h_1),\dots,\mu_0(h_n))\in\mathbb Z_{\ge 0}^n. 
\end{equation}
Also, given $i\in\hat I$ and $\xi\in\hat P^+$, set
\begin{equation}\label{e:formwt}
	f_{i,\xi}(\mu) = \frac{1}{4} (2(\omega_i,\omega_i) - (\bar\xi,\bar\xi) -4\xi(d)-(\mu,\mu)). 
\end{equation}
Note (cf. \eqref{e:cbcomb})
\begin{equation}\label{e:closedball}
	f_{i,\xi}(\mu)\ge 0 \quad\Leftrightarrow\quad (\mu,\mu)\le 2(\omega_i,\omega_i) - (\bar\xi,\bar\xi)-4\xi(d). 
\end{equation}
The following result will be proved in \cref{ss:typeAnpart}.

\begin{thm}\label{t:typeAnpart}
	If $\lie g$ is of type $A_n$, 
	\begin{align*}
		[V(\Lambda_0)\otimes V(\Lambda_i):V(\xi)] &= \sum_{\mu\in\Gamma_\xi} \rho_{\bs \mu_0}(f_{i,\xi}(\mu)) \quad\text{for all}\quad i\in\hat I,\ \xi\in\hat P^+.
	\end{align*}  
\endd
\end{thm}

Since $\Gamma_\xi$ is a discrete set, it follows that the summation appearing in \cref{t:typeAnpart} is actually finite. 
In the case $n=1$, after a more explicit description of the sets $\Gamma_\xi$, this finite set was precisely described in \cite[Proposition 2.6.2 and Corollary 2.6.3]{jamo:exf}. For higher rank, we shall address the task of describing $\Gamma_\xi$ in a more concrete way in \cref{ss:orbits}. In particular, setting 
\begin{equation*}
	\eta = \Lambda_0+\Lambda_i-\xi = \sum_{l\in\hat I}\eta_l\alpha_l \quad\text{and}\quad \bs\eta=(\eta_0,\dots,\eta_n),
\end{equation*}
we will establish in \cref{ss:maincomb} the following connection with \eqref{e:maincomb}:
\begin{equation}\label{e:GxiAngen}
	\xi \equiv \Lambda_j+\Lambda_k \quad\Rightarrow\quad [V(\Lambda_0)\otimes V(\Lambda_i):V(\xi)] = \sum_{(\bs m,\bs p)\in\Gamma_{j,k}} \rho_{\bs b(\bs m,\bs p)} \left( f_{i,\bs\eta}(\bs a(\bs m,\bs p))\right).
\end{equation}
On the other hand, Theorem 2.1 of \cite{mw:tpgen} says
\begin{equation}\label{e:MWtm}
	[V(\Lambda_0)\otimes V(\Lambda_i):V(\xi)] = \tau^n_i(\bs\eta),
\end{equation}
thus completing the proof of \cref{t:maincomb}. 

\begin{ex}
	For $n=2$, $[V(\Lambda_0)\otimes V(\Lambda_1):V(2\Lambda_2-\eta\delta)]$ is given by the right-hand side of \eqref{e:maincomb22}. \endd
\end{ex}

We also recall that one can compute the outer multiplicities of $V(\Lambda_i)\otimes V(\Lambda_j)$ by combining \cref{t:typeAnpart} with \cite[Equation (2.5.1)]{jamo:exf}. More precisely, given $\xi\in\hat P^+,\xi\le \Lambda_i+\Lambda_j$, set
\begin{equation*}
	\xi' = \sum_{k\in I} \xi(h_k)\Lambda_{i+k},
\end{equation*}
and let $c_k\in\mathbb Z, k\in I$, be defined by
\begin{equation*}
	\begin{bsmallmatrix} c_1\\ \vdots\\ c_n \end{bsmallmatrix} = C^{-1}B
\end{equation*} 
where $C$ is the Cartan matrix of type $A_n$ and $B$ is the column matrix whose $k$-th entry is
\begin{equation*}
	\delta_{i,k}+\delta_{j,k}-\xi(h_k) - (\delta_{k,1}+\delta_{k,n})\xi(d).
\end{equation*}
Then, setting $\Lambda_k=\Lambda_i$ if $k\equiv i\in\hat I$ for $k\in\mathbb Z$, \cite[Equation (2.5.1)]{jamo:exf} says
\begin{equation}\label{e:tpgenA}
	[V(\Lambda_i)\otimes V(\Lambda_j):V(\xi)] = [V(\Lambda_0)\otimes V(\Lambda_{i+j}):V(\xi'-c_i\delta)]. 
\end{equation}

\section{Affine Weyl Group Orbits and the Proof of \cref{t:maincomb}}\label{s:orbits}

\subsection{Socle Computations}\label{ss:soc}
As a first step towards obtaining a more explicit characterization of the sets $\Gamma_\xi$, we present a formula for $\soc(\xi)$ for all $\xi\in\hat P$.

Fix $\ell\in\mathbb Z_{>0}$ and $\mu\in P$, let 
\begin{equation*}
	\soc(\ell,\mu) = \soc(\ell\Lambda_0+w_o\mu).
\end{equation*}
and set
\begin{equation}\label{e:defai}
	a_i = -\sum_{j=i}^n (w_o\mu)(h_j) = \sum_{j=1}^{n+1-i}\mu(h_j)\ \ \text{for}\ \ i\in I.
\end{equation}
In particular, 
\begin{equation}\label{e:inca}
	\mu\in P^+ \quad \Leftrightarrow\quad a_1\ge a_2\ge \cdots\ge a_n\ge 0.
\end{equation}
In any case, there exist unique $p_i \in \mathbb{Z},m_i \in \mathbb Z,0<m_i\leq \ell, i\in I$, such that
\begin{align} \label{e:lambdaeps}
	a_i = p_i  \ell + m_i.
\end{align}
Note that,  $\mu\in P^+$ if and only if (cf. \eqref{e:pfordomcomb})
\begin{equation}\label{e:pfordom}
	p_i\ge -1 \ \forall\ i\in I, \quad p_i=-1 \ \Leftrightarrow\ a_i=0 \Rightarrow\ m_i=\ell, \quad\text{and}\quad  p_1\ge p_2\ge\cdots\ge p_n. 
\end{equation}
Recall \eqref{e:resp} and set
\begin{equation}\label{e:lambdaepstup}
	\bs m(\ell,\mu) = (m_1,\dots, m_n), \quad \bs p(\ell,\mu) = (p_1,\dots,p_n),  \quad\text{and}\quad p(\ell,\mu) = \res(\bs p(\ell,\mu)).
\end{equation}
To shorten notation, we henceforth denote these elements by $\bs m,\bs p$, and $p$, respectively. Recall that $\overline{\bs m}$ denotes the associated partition of $|\bs m|$ and note
\begin{equation*}
	\overline{\bs m}\in \mathcal P_\ell^{=n}.
\end{equation*}
Write $\overline{\bs m}=(k_1^{s_1}\ k_2^{s_2}\ \cdots\ k_l^{s_l})$ as in \eqref{e:partmulnot}, set $k_{l+1}=0$ for convenience, consider $\bs s = (s_1,\dots,s_l)\in\mathbb Z^s_{>0}$, and recall \eqref{e:|part|}. 
We will see in \cref{ss:socle} that
\begin{equation}\label{et:soclek}
	\soc(\ell, \mu) \equiv (\ell-k_1)\Lambda_p + \sum_{j=1}^l (k_j-k_{j+1})\Lambda_{p-|\bs s|_j}.
\end{equation}
Here, as before, given $k\in\mathbb Z$, we understand $\Lambda_k=\Lambda_i$ if $k\equiv i\in\hat I$. Equivalently, letting
\begin{equation}
	\bs m'(\ell,\mu) = \overline{(\ell,\bs m)},
\end{equation}
writing $\bs m'(\ell,\mu) = (m'_1,\dots,m'_{n+1})$, and setting $m'_{n+2}=0$ for convenience, \eqref{et:soclek} can be rewritten as
\begin{equation}\label{et:socle'}
	\soc(\ell, \mu) \equiv \sum_{j=0}^n (m'_{j+1}-m'_{j+2})\Lambda_{p-j}.
\end{equation}

\subsection{Orbits Characterizations}\label{ss:orbits}
We now use \eqref{et:socle'} to give a combinatorial characterization of the sets $\Gamma_\xi, \xi\in\hat P^+$. Using either \eqref{et:soclek} or \eqref{et:socle'}, one easily sees that
\begin{equation}\label{e:minpart}
	\soc(\ell, \mu)(h_{p+1}) = \min\{m_i:i\in I\}>0,
\end{equation}
where $p=p(\ell,\mu)$ and $m_i, i\in I$, are defined as in \eqref{e:lambdaeps}, and, for $k\in\mathbb Z$, we understand $h_k=h_i$ if $k\equiv i\in\hat I$. Also, writing $\bs m'(\ell,\mu) = (m'_1,\dots,m'_{n+1})$ as before and recalling \eqref{e:difp}, we have
\begin{equation}\label{e:caredifs}
	(m'_1-m'_2, \dots, m'_{n}-m'_{n+1},m'_{n+1}) = \bs m'(\ell,\mu)^-. 
\end{equation}

Fix $\xi\in\hat P^+$ and let 
\begin{equation}
	\bs c(\xi) = (\xi(h_0),\dots,\xi(h_n))\in\mathbb Z_{\ge 0}^{n+1} \quad\text{and}\quad \ell=\xi(c)=|\bs c(\xi)|.
\end{equation}
Given $p\in\hat I$, let $\sigma_p\in\Sigma_{n+1}$ be the unique element such that 
\begin{equation}
	\sigma_p(j) \equiv p+2-j  \ \ \text{for}\ \ 1\le j\le n+1.
\end{equation}
Recall \eqref{e:socdeg'}. 
In light of \eqref{et:socle'} and \eqref{e:caredifs}, for $\mu\in P^+$, we have
\begin{equation}\label{e:the orbits}
	\ell\Lambda_0+w_o\mu+r\delta \in\widehat{\mathcal W}\xi \quad\Leftrightarrow\quad
	\bs m'(\ell,\mu)^- = \sigma_{p(\ell,\mu)}\bs c(\xi) \ \ \text{and}\ \ r= r(\mu,\xi).
\end{equation}
In other words,
\begin{equation}\label{e:Gxipart}
	\Gamma_\xi = \{\mu\in P^+: \bs m'(\ell,\mu)^- = \sigma_{p(\ell,\mu)}\bs c(\xi) \}. 
\end{equation}

Note \eqref{e:lambdaeps} gives rise to a bijective map
\begin{equation*}
	\phi: P\to [1,\ell]^n\times \mathbb Z^n, \quad \mu\mapsto (\bs m(\ell,\mu),\bs p(\ell,\mu)). 
\end{equation*}
Let $\mathscr P^+=\phi(P^+)$. It follows that $(\bs m,\bs p)\in\mathscr P^+$ if and only if \eqref{e:pfordom} holds (for $\ell=2$, this coincides with the definition of $\mathscr P^+$ given just before \eqref{e:pfordomcomb}).  On the other hand, we have a map
\begin{equation*}
	\psi: \mathcal P_\ell^{=n}\times \hat I \to \hat P^+,
\end{equation*}
where $\psi(\bs m,p)$ is the unique element of $\hat P^+$ satisfying
\begin{equation*}
	\sigma_p\,\bs c(\psi(\bs m,p)) = (\ell,\bs m)^- \quad\text{and}\quad \psi(\bs m,p)(d)=0. 
\end{equation*}
Given $\xi\in\hat P^+$, set
\begin{equation}
	\mathscr C_\xi = \psi^{-1}(\xi-\xi(d)\delta) \quad\text{and}\quad 
	\Gamma'_\xi = \{(\bs m,\bs p)\in\mathscr P^+: (\overline{\bs m},\res(\bs p))\in\mathscr C_\xi\}.
\end{equation}
It now follows from \eqref{e:Gxipart} that the map
\begin{equation}\label{e:Gxipart'}
	\Gamma_\xi \to \Gamma'_\xi, \quad \mu\mapsto (\bs m(\ell,\mu),\bs p(\ell,\mu))
\end{equation}
is a bijection. 
Let us now give more concrete characterizations of $\Gamma_\xi$ and $\Gamma_\xi'$ on special cases.

\begin{ex}\label{ex:sup2soc} 
	Let $a,b\in\mathbb Z_{\ge 0}, a\le b$, and consider $\xi=a\Lambda_j+b\Lambda_k$ for some $j,k\in\hat I$. We start with a generalization of the discussion made in the paragraph starting with \eqref{e:ms2}. 	
	
	Given $1\le s\le n+1$, set
	\begin{equation*}
		\bs m_{a,b}(s) = ((a+b)^{s-1},a^{n+1-s})\in\mathcal P^n_{a+b}.
	\end{equation*}  
	Let also $\mathcal O(s)$ be the orbit of $\bs m(s)$ in $\mathbb Z^n$ under the usual action of $\Sigma_n$ and 
	\begin{equation*}
		\mathcal S_\xi = \{s\in[1,n+1]: s\equiv j'-k' \text{ for all } \{j',k'\}=\{j,k\} \text{ such that } \xi(h_{j'})=a\}. 
	\end{equation*}
	Note
	\begin{equation*}
		a<b \quad\Rightarrow\quad \mathcal S_\xi = \{s\in[1,n+1]: s\equiv j-k \}
	\end{equation*}
	has one element, while
	\begin{equation*}
		a=b \quad\Rightarrow\quad \mathcal S_\xi = \{s\in[1,n+1]: s\equiv \pm|j-k| \},
	\end{equation*}
	may have two elements (cf. \eqref{e:parpar2}). Also,
	\begin{equation*}
		s\in\mathcal S_\xi \quad\Rightarrow\quad \text{there exists unique}\ \ p\in\hat I \ \ \text{such that}\ \ \xi(h_{p+1})=a \ \ \text{and} \ \ \xi(h_{p+1-s})=b,
	\end{equation*}
	which we denote by $p(s)$. More precisely,
	\begin{equation}
		a<b \quad\Rightarrow\quad p(s)+1\equiv j
	\end{equation}
	and $p(s)$ is given by \eqref{e:ps2} otherwise.
	We will show
	\begin{equation}\label{e:G'xi}
		\Gamma'_\xi = \{(\bs m,\bs p)\in\mathscr P^+: \bs m\in\mathcal O(s) \text{ and }\res(\bs p)=p(s) \text{ for some } s\in\mathcal S_\xi\}. 
	\end{equation}
	Equivalently, 
	\begin{equation}\label{e:orbit2}
		\mu\in\Gamma_\xi \quad\Leftrightarrow\quad 
		\bs m'(a+b,\mu) = ((a+b)^s,a^{n+1-s}) \ \ \text{and}\ \ p(a+b,\mu)=p(s) \ \ \text{for some}\ \ s\in\mathcal S_\xi.
	\end{equation}
		
	If $\mu\in\Gamma_\xi$, it follows from \eqref{et:soclek} that $\bs m'(a+b,\mu)$ has at most two parts and, hence, $\bs m'(a+b,\mu)=((a+b)^s,a^{n+1-s})$ for some $1\le s\le n+1$. In that case, \eqref{et:socle'} reads
	\begin{equation}\label{et:socle2}
		\soc(a+b, \mu) \equiv a\Lambda_{p+1} + b\Lambda_{p+1-s} \quad\text{where}\quad p = p(a+b,\mu),
	\end{equation}
	from where the right-hand-side of \eqref{e:orbit2} easily follows. The converse is now also clear. 
	
	In the case $a=b$, \eqref{e:G'xi} can be simplified as follows. Set $\ell=a+b$, so $\xi=\ell\Lambda_j$ for some $j\in\hat I$. In this case, we will show
	\begin{equation}\label{e:elli}
		\mu\in \Gamma_\xi \quad\Leftrightarrow\quad 
		\mu \in \ell P^+ \ \ \text{and}\ \  
		j+\sum_{i\in I}(n+1-i)\frac{\mu(h_i)}{\ell} \in(n+1)\mathbb Z.
	\end{equation}
	Note that, in this case $\mathcal S_\xi=\{n+1\}$ and, hence, $\bs m_{a,a}(s) = (\ell^n)$. Moreover, \eqref{e:orbit2} becomes 
	\begin{equation}\label{e:socell}
		(\bs m,\bs p)\in \Gamma_\xi' \quad\Leftrightarrow\quad \bs m=(\ell,\dots,\ell) \quad\text{and}\quad  \res(\bs p)+1\equiv j.
	\end{equation}
	
	Since 
	\begin{equation*}
		\mu\in\Gamma_\xi \quad\Leftrightarrow\quad (\bs m(\ell,\mu),\bs p(\ell,\mu))\in\Gamma'_\xi,
	\end{equation*}
	we see the first condition in \eqref{e:elli} must be satisfied by all $\mu\in\Gamma_\xi$. Letting $a_i$ be as in \eqref{e:defai} and writing $\bs p(\ell,\mu)=(p_1,\dots,p_n)$, this implies 
	\begin{equation*}
		p_i = \frac{a_i}{\ell}-1.
	\end{equation*}
	Letting $\bs a = (a_1,\dots,a_n)$, it easily follows from \eqref{e:defai} that
	\begin{equation*}
		(\mu(h_n),\dots,\mu(h_1)) = \bs a^-.
	\end{equation*}
	Letting $p=\res(\bs p(\ell,\mu))$, then $p \equiv -\sum_{i\in I}p_i$ by definition. The second condition in \eqref{e:socell} then implies
	\begin{equation*}
		j \equiv 1 - \sum_{i\in I}p_i = 1 - \sum_{i\in I}\left(\left(\sum_{j=1}^{n+1-i}\frac{\mu(h_j)}{\ell}\right)-1\right) = n+1 - \sum_{i\in I} (n+1-i)\frac{\mu(h_i)}{\ell},
	\end{equation*}
	which proves the second condition on the right-hand-side of \eqref{e:elli} holds. 
	
	Conversely, if the first condition on the right-hand side of \eqref{e:elli} holds, we have $m_i=\ell$ for all $i\in I$. The above computations then show the second condition on the right-hand-side of \eqref{e:elli} implies $p+1\equiv j$, thus showing $(\bs m(\ell,\mu),\bs p(\ell,\mu))\in\Gamma'_\xi$. 	\endd
\end{ex}

\subsection{On the Proof of \eqref{et:soclek} and Related Computations}\label{ss:socle}
Denote the elements defined in \eqref{e:lambdaepstup} by $\bs m,\bs p$, and $p$, respectively, set $m_{n+1}=0$, and let $\sigma\in\Sigma_{n+1}$ be such that 
\begin{equation}\label{e:defsigmaA}
	m_{\sigma(k)}\le m_{\sigma(k+1)} \ \ \text{for all}\ \ 1\le k\le n, k\ne p, \ \ \text{and}\ \ m_{\sigma(n+1)}\le m_{\sigma(1)} \ \ \text{if}\ \ p>0.
\end{equation}
Note 
\begin{equation}\label{e:mmaxmin}
	m_{\sigma(p+1)}=m_{n+1}=0 \ \ \text{and}\ \ m_{\sigma(p)} = \max\{m_i:i\in I\}>0. 
\end{equation} 
Extend $\sigma$ to $\hat I\cup\{n+1\}$ by setting $\sigma(0)=0$ and set $m_0 = m_{\sigma(n+1)}$. 
The following is a mildly modified version of \cite[Proposition 3.2]{kura} from where \eqref{et:soclek} is easily deduced. 
\begin{equation}\label{et:socle}
	\soc(\ell, \mu) \equiv (\ell-m_{\sigma(p)})\Lambda_p + \sum_{i\in\hat I\setminus\{p\}} (m_{\sigma(i+1)}-m_{\sigma(i)})\Lambda_i. 
\end{equation}
The argument we present here is slightly modified from that of \cite{kura} and makes use of computations which explain Definition \eqref{e:fpart}. The main difference in the argument here compared to that of \cite{kura} is that we do not use Dynkin diagram automorphisms, while  \cite{kura} regards $\mathcal{\widehat W}$ as a subgroup of the extended affine Weyl group $\mathcal{T} \ltimes \mathcal{\widehat W}$, where $\mathcal{T}$ is the group of automorphisms of the affine Dynkin diagram, and the proof indeed utilizes elements from $\mathcal T$. Aiming at extending the results of \cite{kura}, as well as those of \cref{ss:omdf},  we were able to adapt the argument below to obtain generalizations of \eqref{et:socle} for types $B,C,D,G$ which will appear elsewhere. 

It will be convenient to introduce further notation. Set
\begin{equation}\label{e:defuiA}
	u_i = r_i\ell-m_i \ \ \text{for}\ \ 0\le i\le n+1,
\end{equation}
with
\begin{equation}\label{e:defriA}
	r_{\sigma(i)} =   1 \ \ \text{if}\ \ 0\leq i \leq p \ \ \text{and}\ \ r_{\sigma(i)} = 0 \ \ \text{otherwise.}
\end{equation}
In particular, $u_{n+1}=u_{\sigma(p+1)}=0$. 

It will also be convenient to work with a different basis for $\lie h^*$, which we now recall (cf. \cite[Chapter 12]{humlie}). The convenience arises from the fact that the action of the Weyl group on this basis lead to computations which are more easily handled in comparison to working with the bases of simple roots or fundamental weights. Thus, let $\lie h_{\mathbb R}^*$ denote the real span of $\Delta$, consider  the standard basis $\{\varepsilon_i:1\leq i \leq n+1\}$ of $\mathbb R^{n+1}$, and set
\begin{align} \label{e:rootsystemrealizationA}
	v=	\varepsilon_1 + \cdots + \varepsilon_{n+1}
\end{align} 
Letting $\overline{\varepsilon_i}= \varepsilon_i + \mathbb{R} v$, it follows that $\{\overline{\varepsilon}_i: 1\le i\le n\}$ is a basis of $\mathbb R^{n+1}/\mathbb{R}v$. Therefore, there exists a linear isomorphism $\iota: \lie h_{\mathbb R}^* \to \mathbb R^{n+1}/\mathbb{R}v$ determined by
\begin{equation}\label{e:convbasisA}
	\iota(\alpha_i) = \overline{\varepsilon_i} - \overline{\varepsilon_{i+1}}, \quad \text{for} \quad i<n, \quad \text{and} \quad     \iota(\alpha_n) = 
	\overline{\varepsilon_n}-\overline{ \varepsilon_{n+1}}.
\end{equation}
The preimage of  $\{\overline{\varepsilon}_i: 1\le i\le n\}$ by $\iota$ is then a basis of $\lie h_\mathbb R^*$ and, hence, also of $\lie h^*$. By abuse of notation, we denote $\iota^{-1}(\overline{\varepsilon}_i), 1\le i\le n+1$, simply by $\varepsilon_i$. Quite clearly,
\begin{equation}\label{e:ratcoord}
	\lambda = \sum_{i\in I} a_i\varepsilon_i \in P \quad\Rightarrow\quad a_i\in\mathbb Q. 
\end{equation}
Indeed, it will eventually follow that
\begin{equation}\label{e:lambdainepsilonA}
	a_i = \sum\limits_{j\in I}\lambda(h_j).
\end{equation}

Henceforth, we use the above with $\lambda=-w_o\mu$, so  \eqref{e:lambdainepsilonA}  becomes \eqref{e:defai}. 
Recall \eqref{e:lambdaeps}, \eqref{e:defriA}, the definition of the map $\tau$ given by \eqref{e:areflex}, and consider 
\begin{equation} \label{e:tdef}
	t:= t_{(p_1+r_1) \varepsilon_1 + \cdots + (p_n+r_n)\varepsilon_n} \in \tau( \hlie h ^*).
\end{equation}
Recalling \eqref{e:moddelta} as well, it is immediate from the definition of $\tau$ that 
\begin{equation*}
	t(\ell\Lambda_0) \equiv \ell\Lambda_0 + \ell\sum_{i \in I } (p_i+r_i) \varepsilon_i \quad\text{and}\quad t\lambda\equiv\lambda.
\end{equation*}
Using \eqref{e:lambdaeps}, it then follows that  
\begin{align*}
	t(\ell\Lambda_0-\lambda) &\equiv\ell \Lambda_0+ \sum_{i \in I }\left(r_i\ell - m_i \right) \varepsilon_i.
\end{align*}
Recall  \eqref{e:defuiA} and the definition of the permutation $\sigma$ in the paragraph containing \eqref{e:mmaxmin},   and set 
\begin{equation}\label{e:barsoc}
	\Lambda = \ell \Lambda_0 + \sum_{i=1}^{n+1}u_{\sigma(i)} \varepsilon_{i} = \ell\Lambda_0 + \sum_{i \in I }u_i \varepsilon_{\sigma^{-1}(i)},
\end{equation}
where,  the last equality follows  since $u_{n+1}=u_{\sigma(p+1)}=0$. 
Consider also $\tilde \sigma \in GL ( \lie h ^*_{\mathbb{R}}) $ given by
\begin{equation}\label{e:tilsigma}
	\tilde \sigma(\varepsilon_i) =  \varepsilon_{\sigma^{-1}(i)} \quad\text{for}\quad 1\le i\le n+1,
\end{equation}  
and note 
\begin{align*}
	\tilde\sigma t(\ell\Lambda_0-\lambda) &\equiv \ell \Lambda_0  + \sum_{i \in I }\left(r_i\ell - m_i \right) \varepsilon_{\sigma^{-1}(i)} \overset{\eqref{e:defuiA}}{=} \Lambda.
\end{align*}
We claim
\begin{gather} 
	\label{e:tinT} t \in \tau(Q),\\
	\label{e:tildesigmainW} \tilde \sigma \in \mathcal W,\\
	\label{e:ui-ujinZ} u_{\sigma(i)} - u_{\sigma(i+1)} \in \mathbb{Z}_{\geq 0},\\ \label{e:Lambda=sum}\Lambda=\sum_{i \in \hat I} (u_{\sigma(i)}-u_{\sigma(i+1)}) \Lambda_{i}. 
\end{gather}
These claims imply $\Lambda \in (\widehat{\mathcal W}(\ell \Lambda_0 - \lambda) \cap \hat P^+)+\mathbb Z\delta$ and, hence, $\Lambda \equiv \soc(\ell,\mu)$. 

Note \eqref{e:lambdainepsilonA} implies $m_i\in\mathbb Z$ for all $1\le i\le n+1$ and, hence, the same is true for $u_i$ by \eqref{e:defuiA}. If $p=0$, then \eqref{e:defriA} and \eqref{e:defuiA} imply $u_i=-m_{i}$ for all $1\le i\le n+1$. Since $m_0 = m_{\sigma(n+1)}$, it follows from \eqref{e:defsigmaA} that  
\begin{align*}
	u_{\sigma(0)} = \ell - m_{\sigma(n+1)} \geq 0 = u_{\sigma(1)} \geq u_{\sigma(2)}\geq \cdots \geq  u_{\sigma(n)} \geq u_{\sigma(n+1)}.
\end{align*}
On the other hand, if $p>0$, then 
\begin{align*}
	u_{\sigma(0)}= \ell - m_{\sigma(n+1)} , \quad u_{\sigma(i)}=\ell  -m_{\sigma(i)}, \ 1\leq i \leq p, \quad \text{and} \quad u_{\sigma(i)}=  -m_{\sigma(i)}, \ p< i \leq n+1,
\end{align*}
and \eqref{e:defsigmaA} implies $\ell\geq m_{\sigma(p)} \geq \cdots \geq m_{\sigma(1)} \geq  m_{\sigma(n+1)} \geq \cdots \geq m_{\sigma(p+1)}= 0$.
Thus, for any value of $p$, we see that, for all $i\in\hat I$, we have 
\begin{align}\label{e:soccefu}
	u_{\sigma(i)}-u_{\sigma(i+1)}=\begin{cases}
		\ell-m_{\sigma(p)}, & \text{if } i=p\\
		m_{\sigma(1)}-m_{\sigma(n+1)}, & \text{if } i=0\ne p\\ 
		m_{\sigma(i+1)}-m_{\sigma(i)}, & \text{otherwise,} 
	\end{cases} 
\end{align}
from where \eqref{e:ui-ujinZ} follows. Moreover, \eqref{e:soccefu} and \eqref{e:Lambda=sum} clearly imply \eqref{et:socle}, thus completing the proof. 

Thus, it remains to check the other three claims.  Since $\alpha_i = \varepsilon_i-\varepsilon_{i+1}$ by \eqref{e:convbasisA}, we have
\begin{equation}\label{e:fwepsA}
	\begin{aligned}
		& \omega_i  =  \sum_{j=1}^{i} \frac{(n+1-i)j}{n+1} \alpha_j + \sum_{j=i+1}^{n} \frac{i(n+1-j)}{n+1} \alpha_j = \varepsilon_1 + \cdots + \varepsilon_i,\\  
		& \Lambda_i = \omega_i + \Lambda_0, \quad i\in I, \quad\text{and}\quad \alpha_0=\delta - (\varepsilon_1-\varepsilon_{n+1}).
	\end{aligned}
\end{equation}
The following observation then leads to \eqref{e:lambdainepsilonA}:
\begin{equation} \label{e:lambdaA}
	\sum_{i \in I}a_i \varepsilon_i = \lambda = \sum_{i \in I} \lambda(h_i) \omega_i \overset{\eqref{e:fwepsA}}{=}  \sum_{i \in I} \left(\sum_{j =i}^n \lambda(h_j)\right) \varepsilon_i.
\end{equation}
Recalling that $\omega_{n+1}=\omega_0=0$ and $\Lambda_{n+1}=\Lambda_0$, we see 
\begin{equation}\label{e:fwepsC'}
	\varepsilon_i =\omega_i-\omega_{i-1} =\Lambda_i-\Lambda_{i-1}  = 
	\sum_{j=i}^{n} \frac{n+1-j}{n+1}\alpha_j - \sum_{j=1}^{i-1} \frac{j}{n+1}\alpha_j \quad\text{for all}\quad 1\leq i\leq n+1.
\end{equation}
Then, given $b_i\in\mathbb Z, i\in I$, note
\begin{align*}
	b_1 \varepsilon_1 + \cdots +b_n \varepsilon_n 
	&=   \sum_{i=1}^{n} b_i \Big( \sum_{j=i}^{n} \frac{n+1-j}{n+1} \alpha_j  - \sum_{j=1}^{i-1} \frac{j}{n+1} \alpha_j \Big)\\
	&=\sum_{i=1}^n \left( \frac{n+1-i}{n+1} \sum_{j=1}^i b_j -\frac{i}{n+1} \sum_{j=i+1}^n b_j  \right) \alpha_i.
\end{align*}
Let $c_i$ be the coefficient of $\alpha_i$ in the last expression. One easily checks $c_i\in\mathbb Z$ for all $i\in I$ if and only if $b_1+\cdots+b_n \in(n+1)\mathbb Z$. In other words, 
\begin{equation}\label{e:epsinQveeA}
	b_1 \varepsilon_1 + \cdots +b_n \varepsilon_n \in Q \quad\Leftrightarrow\quad  b_1+\cdots +b_n \in (n+1)\mathbb Z .
\end{equation}
Since
\begin{equation*}
	\sum_{i\in I} p_i+r_i = |\bs p| + p \in(n+1)\mathbb Z,
\end{equation*}
it follows from \eqref{e:defriA} and \eqref{e:tdef} that $t \in \tau(Q)$, thus proving \eqref{e:tinT}. 

The Weyl group of type $A$ is isomorphic to $\Sigma_{n+1}$ with action of $\mathcal{W}$ on the vectors $\varepsilon_i$ given as follows \cite[Chapter 12]{humlie}. Let $\sigma_i\in \Sigma_{n+1}$ be the transposition $(i,i+1)$. Then,
\begin{equation}\label{e:weylonepsA}
	s_i\varepsilon_j = \varepsilon_{\sigma_i(j)} \ \ \text{for}\ \ i\in I, 1\leq j \leq n+1.
\end{equation}
In particular,  \eqref{e:tilsigma} implies \eqref{e:tildesigmainW}.

Finally, since $\varepsilon_{n+1}=-\sum_{i \in I} \varepsilon_i$, \eqref{e:Lambda=sum} follows by noticing
\begin{align*}
	\Lambda = \ell \Lambda_0 + \sum_{i=1}^{n+1} u_{\sigma(i)} \varepsilon_{i}&=\ell \Lambda_0 + \sum_{i \in I }(u_{\sigma(i)}-u_{\sigma(n+1)}) \varepsilon_{i}\\
	& \stackrel{\eqref{e:fwepsA}}{=} \ell \Lambda_{0} + \sum_{i \in I }(u_{\sigma(i)}-u_{\sigma(n+1)}) (\Lambda_i-\Lambda_{i-1} ) =
	\sum_{i \in \hat I} (u_{\sigma(i)}-u_{\sigma(i+1)}) \Lambda_{i},
\end{align*}
where the last step follows from \eqref{e:defuiA} since $m_0 = m_{\sigma(n+1)}$.

Let us also use \eqref{e:fwepsA} for checking that
\begin{align} \label{e:(i,j)A}
	(\varepsilon_i,\varepsilon_j)=
	\frac{n}{n+1}\delta_{i,j}+ \frac{1}{n+1}(\delta_{i,j}-1) \quad\text{for all}\quad i,j\in I,
\end{align}
which will be needed for proving \eqref{e:f=ip} below, explaining Definition \eqref{e:fpart}. 
Indeed,
\begin{align*}
	(\varepsilon_i,\varepsilon_i)= \left(\omega_i-\omega_{i-1}, \sum_{j=i}^{n} \frac{n+1-j}{n+1}\alpha_j - \sum_{j=1}^{i-1} \frac{j}{n+1}\alpha_j \right)=\frac{i-1}{n+1}+\frac{n+1-i}{n+1}=\frac{n}{n+1},
\end{align*}
for all $i\in I$ and, for $1\le i<j\le n$,
\begin{align*}
	(\varepsilon_i,\varepsilon_j)=\left(\omega_i-\omega_{i-1},\sum_{k=j}^{n} \frac{n+1-k}{n+1}\alpha_k - \sum_{k=1}^{j-1} \frac{k}{n+1}\alpha_k \right)=\frac{i-1}{n+1}-\frac{i}{n+1}=-\frac{1}{n+1}.
\end{align*}

\subsection{Proof of \eqref{e:GxiAngen}}\label{ss:maincomb}
In light of \cref{t:typeAnpart} and \eqref{e:MWtm}, we need to establish a comparison between the right-hand sides of both expressions. This will be accomplished using a special case of \cref{ex:sup2soc}. Fix $i\in\hat I$ and let $\xi\in\hat P^+$. Then,
\begin{equation*}
	[V(\Lambda_0):V(\Lambda_i):V(\xi)] \ne 0 \quad\Rightarrow\quad \Lambda_0+\Lambda_i-\xi = \sum_{r\in\hat I} \eta_j\alpha_j \quad\text{for some}\quad \eta_j\in\mathbb Z_{\ge 0}.
\end{equation*}
Note this immediately implies $\eta_0=-\xi(d)$. Moreover, the condition $\xi\in\hat P^+$ is equivalent to \eqref{e:domceta}. Thus, we have an injective map
\begin{equation*}
	\{\xi\in\hat P^+:\xi\le \Lambda_0+\Lambda_i\}\to \mathbb Z^{n+1}_{\ge 0}, \ \ \xi\mapsto \bs\eta(\xi) = (\eta_0,\eta_1,\dots,\eta_n),
\end{equation*}
and, by \eqref{e:xi=jk}, $\xi$ must be of the form $\Lambda_j+\Lambda_k-\eta_0\delta$ for some $j,k\in\hat I$. Thus, \cref{t:maincomb} follows if we check the right-hand side of the formula given in \cref{t:typeAnpart} coincides with the right-hand side of \eqref{e:GxiAngen} whenever $\xi=\Lambda_j+\Lambda_k-\eta_0\delta$. 

A quick comparison of \eqref{e:G'xi} with $a=b=1$ with \eqref{e:Gjk} shows
\begin{equation*}
	\Gamma'_\xi = \Gamma_{j,k}.
\end{equation*} 
Given $\mu\in\Gamma_\xi$, Write  $\bs m(2,\mu)=(m_1,\dots,m_n), \bs p(2,\mu)=(p_1,\dots,p_n)$ and let $\bs a$ be as in \eqref{e:acomb} or, equivalently, \eqref{e:lambdaeps}. Then,
\begin{equation*}
	\mu(h_r) = a_r-a_{r+1} = 2(p_r-p_{r+1}) + m_r-m_{r+1} \quad\text{for all}\quad r\in I 
\end{equation*}
where we set $a_{n+1}=0, p_{n+1}=-1$ and $m_{n+1}=2$. Plugging this  back in \eqref{e:mu01} and \eqref{e:bsmuj} and comparing with \eqref{e:boundcomb}, one easily checks that 
\begin{equation*}
	\bs\mu_0 = \bs b(\bs m(2,\mu),\bs p(2,\mu)) \quad\text{for all}\quad \mu\in\Gamma_\xi.
\end{equation*}
Moreover, it follows from \eqref{e:(i,j)A} and \eqref{e:lambdainepsilonA} that
\begin{equation}\label{e:f=ip}
	(\mu,\mu) = f(\bs a).
\end{equation}
Another quick comparison of the definitions \eqref{e:formcomb} and \eqref{e:formwt} completes the proof.

\section{Outer Multiplicity Proofs}\label{s:omp}

\subsection{A Parameterizing Set of Reduced Expressions}\label{ss:lrseq}
Assume $\lie g$ is of type $A_n$. We will need  a characterization of reduced expressions of elements in $\widehat{\mathcal W}$. 
Consider the following elements of $\mathcal W$:
\begin{equation}
	s_{i,j} = (s_1s_2\cdots s_i)(s_n\cdots s_{j+1}s_j) \quad\text{for}\quad 0\le i<n,\ 1\le j\le n+1.
\end{equation}
In particular, $s_{0,j}=s_n\cdots s_{j+1}s_j, s_{i,n+1} = s_1s_2\cdots s_i$, and 
\begin{equation*}
	\ell(s_{i,j}) = i+n+1-j.
\end{equation*}
Given $l\in\mathbb Z_{\ge 0},w\in\mathcal W$, $\bs i = (i_1,\dots,i_l), \bs j=(j_1,\dots,j_l)\in\mathbb Z^l$ with $0\le i_s<n$ and $1\le j_s\le n+1$ for all $1\le s\le d$,  consider
\begin{equation}
	\sigma(w,\bs i,\bs j) = ws_0s_{i_1,j_1}s_0s_{i_2,j_2}\cdots s_0s_{i_l,j_l}\in\widehat{\mathcal W}. 
\end{equation} 
The number $l$ will be referred to as the length of the pair $(\bs i,\bs j)$ and will be denoted by $\ell(\bs i,\bs j)=l$. We shall say the length-$l$ pair $(\bs i,\bs j)$ is reduced if the following conditions are satisfied.
\begin{enumerate}[(1)]
	\item If $s<l$, either $(i_s,j_s)=(0,1)$ or $i_s\ne 0$ and $j_s\ne n+1$;
	\item $\bs i$ is non-increasing and $\bs j$ is non-decreasing;
	\item $s<l$ and $i_s < j_s -1 \quad\Rightarrow\quad i_s > i_{s+1}$;
	\item $1<l$ and $i_s < i_s -1 \quad\Rightarrow\quad j_{s-1} < j_s$.
\end{enumerate}
Let $\mathscr R = \{(\bs i,\bs j): (\bs i,\bs j) \text{ is reduced}\}$. 
The following theorem is a rephrasing of \cite[Theorem 2.13]{har} (see also \cite[Section 2.3]{est:msc} and \cite{bur}). 

\begin{thm}\label{t:affredexp}
	If $(\bs i,\bs j)\in\mathscr R$ and $l=\ell(\bs i,\bs j)$, 
	\begin{equation*}
		\ell(\sigma(w,\bs i,\bs j)) = \ell(w)+l+\sum_{k=1}^l \ell(s_{i_k,j_k}) \quad\text{for all}\quad w\in\mathcal W. 
	\end{equation*}
	Moreover, the map $\mathcal W\times \mathscr R\to\widehat{\mathcal W}, (w,\bs i,\bs j)\mapsto \sigma(w,\bs i,\bs j)$ is bijective. \qed
\end{thm}

Recall that, if $(A,\le)$ is a poset, a subset $B\subseteq A$ is said to be cofinal if, for all $a\in A$, there exists $b\in B$ such that $a\le b$. A sequence $a_s, s\ge 0$, in $A$ is said to be cofinal if the set $\{a_s:s\ge 0\}$ is cofinal. The following corollary plays a crucial role in the proof of \cref{t:typeAn}.

\begin{cor} \label{c:inccof} 
	The sequence $w_k=(s_0s_{n-1,1})^k, k\ge 0$, is increasing and cofinal with respect to the Bruhat ordering. In particular, the same is true for the sequence $w_ow_k$.
\end{cor}

\begin{proof} 
	Let $\bs i\in\mathbb Z^k$ be the element such that $i_r=n-1$ for all $1\le r\le k$ and let $\bs j\in\mathbb Z^k$ be the element such that $j_r=1$ for all $1\le r\le k$. Clearly $(\bs i,\bs j)\in\mathscr R$ and
	\begin{equation*}
		w_k = \sigma(e,\bs i,\bs j) \quad\text{while}\quad w_ow_k =\sigma(w,\bs i,\bs j). 
	\end{equation*}
	It then follows from \cref{t:affredexp} that the justaposition of $k$ copies of $s_0$ followed by the fixed reduced expression for $s_{n-1,1}$ is a reduced expression for $w_k$ and similarly for $w_ow_k$. Since $s_i$ appears in the reduced expression for $s_0s_{n-1,1}$ for all $i\in\hat I$, both conclusions now follow immediately. 
\end{proof}

\subsection{Proof of \cref{t:typeAn}}
According to \eqref{t:multrel} and \eqref{e:multrel}, 
\begin{equation*}
	[V:V(\xi)] =  \sum_{\mu\in \Gamma_\xi}  \max_{\gamma\in P^-_\Lambda}\ [D(\gamma): D(\ell+1,\mu,r(\mu,\xi))].
\end{equation*}
In what follows, we shorten notation and write simply $r$ in place of $r(\mu,\xi)$. 

Let $w_k=(s_0s_{n-1,1})^k, k\ge 0$, as in \cref{c:inccof}. A straight forward computation shows
\begin{equation}
	w_1\Lambda = \ell\Lambda_0 + (\lambda+\ell\theta) + (s-\ell-|\lambda|)\delta. 
\end{equation}
Since $w_k=w_1^k$, iterating we get
\begin{equation}
	w_k\Lambda = \ell\Lambda_0 + (\lambda+k\ell\theta) + (s-k(k\ell+|\lambda|))\delta. 
\end{equation}
In particular, it follows that $w_ow_k\in\widehat{\mathcal W}_\Lambda^-$ and, hence,
$\gamma_k:=w_kw_k\Lambda\in P_\Lambda^-$. The latter is equivalent to saying that $(\lambda+k\ell\theta,s-k(k\ell+|\lambda|))\in\Gamma_{\Lambda}$  for all $k\ge 0$.
By \cref{c:inccof}, $w_ow_k$ is increasing with respect to the Bruhat order which, together with \eqref{e:BruhatDem}, implies that
\begin{equation*}
	[D(\gamma_{k+1}): D(\ell+1,\mu,r)]_{\ell+1}\ge [D(\gamma_k): D(\ell+1,\mu,r)].
\end{equation*}
Moreover, since $w_ow_k$ is also cofinal, we conclude
\begin{equation*}
	\max_{\gamma\in P^-_\Lambda}\ [D(\gamma): D(\ell+1,\mu,r)] = \lim_{k\to\infty} [D(\gamma_k): D(\ell+1,\mu,r)].
\end{equation*}
Finally,
\begin{align*}
	[D(\gamma_k): D(\ell+1,\mu,r)] & = [D(\ell,\lambda+k\ell\theta,s-k(k\ell+|\lambda|)): D(\ell+1,\mu,r)]\\
	& = [D(\ell,\lambda+k\ell\theta): D(\ell+1,\mu,r-s+k(k\ell+|\lambda|))],
\end{align*}
thus completing the proof.

\subsection{Quantum Binomials and Partitions}
For the proof of  \cref{t:typeAnpart}, we will need some facts about the interplay between the notion of quantum binomials and partitions. The quantum binomials are defined as
\begin{equation*}
	\qbinom{m}{p}_q = \prod_{n=0}^{p-1} \frac{1-q^{m-n}}{1-q^{p-n}}\in\mathbb Z[q,q^{-1}] \quad\text{for} \quad m,p\in\mathbb Z_{\ge 0}, \ p\le m,
\end{equation*}
where $q$ is a formal variable. Recall that the coefficient of $q^s$ in $\tqbinom{m}{p}_q$ is equal to $\rho_{m-p}^p(s)$, the number of partitions of $s$ in at most $p$ parts all bounded by $m-p$ \cite[Chapter 3]{andrews}. In other words:
\begin{equation}\label{e:qbinpart}
	\qbinom{m}{p}_q = \sum_{s\ge 0}\rho^p_{m-p}(s)q^s.
\end{equation}
Recall \eqref{e:multpart} and suppose $\bs m,\bs p\in\mathbb Z_{\ge 0}^l$ satisfy $p_j\le m_j$ for all $1\le j\le l$. One easily checks using \eqref{e:qbinpart} that 
\begin{equation}\label{e:binmpart}
	\prod_{j=1}^l \qbinom{m_j}{p_j}_q = \sum_{s\ge 0} \rho^{\bs p}_{\bs m-\bs p}(s) q^s. 
\end{equation}

The following lemma about multipartitions will also play a role in the proof of  \cref{t:typeAnpart}. Recall \eqref{e:multpart}  and let $\langle\,,\,\rangle$ denote the usual inner product in $\mathbb R^l$. 

\begin{lem}\label{l:mpart=0}
	Let $m,k\in\mathbb Z_{\ge 0}, \bs a,\bs b\in\mathbb Z^l_{\ge 0}$ for some $l>0$, and set $a=\max\{a_j:1\le j\le l\}$.
	\begin{enumerate}[(a)]
		\item If $k\ge a$, then  $\rho^{k-\bs a}_{\bs b}(m) \ne 0$ only if $m\le k|\bs b| - \langle \bs a,\bs b\rangle$.
		\item  If $f\in\mathbb Z_{\ge 0}$ and $k\ge \max\{f+a,(f+\langle \bs a,\bs b\rangle)/|\bs b|\}$, there exists a bijection $\mathcal P_{\bs b}^{k-\bs a}(k|\bs b| - \langle \bs a,\bs b\rangle-f) \to \mathcal P_{\bs b}(f)$. In particular, $\lim\limits_{k\to\infty} \rho_{\bs b}^{k-\bs a}(k|\bs b| - \langle \bs a,\bs b\rangle-f) = \rho_{\bs b}(f)$. 
	\end{enumerate}
\end{lem}

\begin{proof}
	Let $\bs m\in\mathcal P^{k-\bs a}_{\bs b}(m)$, say 
	\begin{equation*}
		\bs m = (\bs m_1,\dots,\bs m_l) \quad\text{with}\quad \bs m_j\in\mathcal P^{k-a_j}_{b_j}(m_j)
	\end{equation*}
	where $m_j\in\mathbb Z_{\ge 0}$ are such that $m_1+\cdots+m_l=m$. Write 
	\begin{equation*}
		\bs m_j = (m_{j,1},\dots, m_{j,k-a_j}) \ \ \text{for}\ \ 1\le j\le l.
	\end{equation*}
	In particular, 
	\begin{equation*}
		|\bs m_j| = \sum_{s=1}^{k-a_j} m_{j,s} = m_j \quad\text{and}\quad \sum_{j=1}^l\sum_{s=1}^{k-a_j} m_{j,s} =m. 
	\end{equation*}
	However, $m_{j,s}\le b_j$ for all $1\le j\le l, 1\le s\le k-a_j$ and, hence,
	\begin{equation*}
		m=\sum_{j=1}^l\sum_{s=1}^{k-a_j} m_{j,s} \le \sum_{j=1}^l\sum_{s=1}^{k-a_j} b_j = \sum_{j=1}^l (k-a_j)b_j = k|\bs b| - \langle \bs a,\bs b\rangle, 
	\end{equation*}
	thus completing the proof of (a). 
	
	For part (b), we streamline and extend  the argument used for proving \cite[(4.3.8)]{jamo:exf}. Recall \eqref{e:sumntopart} and \eqref{e:barm} and note that, for any $a,b,m\in\mathbb Z_\ge 0$, the map
	\begin{equation}
		\varphi_m^{a,b}:\mathcal P_b^a(m) \to\mathcal P_b^a(ab-m), \quad \bs m\mapsto \overline{b-\bs m}
	\end{equation}
	is invertible and its inverse is $\varphi_{ab-m}^{a,b}$.  The assumption $k\ge (f+\langle \bs a,\bs b\rangle)/|\bs b|$ implies 
	\begin{equation*}
		k|\bs b|-\langle\bs a,\bs b\rangle - f \ge 0.
	\end{equation*}
	One then easily checks that we have a bijective map
	\begin{equation}
		\begin{aligned}
			\bs\varphi:  \mathcal P_{\bs b}^{k-\bs a}(k|\bs b|-\langle\bs a,\bs b\rangle - f) & \to \mathcal P_{\bs b}^{k-\bs a}(f)\\
			(\bs m_1,\dots,\bs m_l) & \mapsto (\varphi_{|\bs m_1|}^{k-a_1,b_1}(\bs m_1),\dots,\varphi_{|\bs m_l|}^{k-a_l,b_l}(\bs m_l)). 
		\end{aligned}
	\end{equation}
	Moreover, note
	\begin{equation*}
		\varphi_{|\bs m_j|}^{k-a_j,b_j}(\bs m_j)\in\mathcal P_{b_j}^{k-a_j}((k-a_j)b_j-|\bs m_j|) \quad\text{for all } 1\le j\le l.
	\end{equation*}
	In particular, 
	\begin{equation*}
		(k-a_j)b_j-|\bs m_j| \le f\le k-a \le k-a_j,
	\end{equation*}
	where we used that $k\ge f+a$ for the second inequality. It then follows from \eqref{e:stabnparts} that 
	\begin{equation*}
		\mathcal P_{b_j}^{k-a_j}((k-a_j)b_j-|\bs m_j|) = \mathcal P_{b_j}((k-a_j)b_j-|\bs m_j|) \quad\text{for all } 1\le j\le l
	\end{equation*}
	and, hence,
	\begin{equation*}
		\mathcal P_{\bs b}^{k-\bs a}(f) = \mathcal P_{\bs b}(f),
	\end{equation*}
	thus completing the proof.
\end{proof}

\subsection{Proof of \cref{t:typeAnpart}}\label{ss:typeAnpart}
Recall \eqref{e:gsdfm} and \eqref{e:mu01} and let $\lambda\in P^+$ be such that $\lambda-\mu\in Q^+$. The following is equation (4.1) of \cite{bcsw:mdpl2}.
\begin{equation}\label{e:multflagAn}
	[D(1,\lambda):D(2,\mu)](q)= q^{\frac{1}{2}(\lambda + \mu_1, \lambda-\mu)} \prod_{j\in I}\tqbinom{(\lambda-\mu, \omega_j)+(\mu_0,\alpha_j)}{(\lambda-\mu,\omega_j)}_q.
\end{equation}

Given $\eta\in  Q$, let $\bs a^\eta\in\mathbb Z^n$ be such that
\begin{equation}\label{e:bsaeta}
	\eta = \sum_{i\in I} a^\eta_i\alpha_i.
\end{equation}
Equivalently, $a^\eta_i=(\eta,\omega_i)$.
Recalling  \eqref{e:bsmuj} and using \eqref{e:binmpart}, the above can be rewritten as 
\begin{equation}\label{e:multflagAnpart}
	[D(1,\lambda):D(2,\mu)](q)= \sum_{s\ge 0}\rho^{\bs a^{\lambda-\mu}}_{\bs\mu_0}(s)\,q^{s+\frac{1}{2}(\lambda + \mu_1, \lambda-\mu)}
\end{equation}
or, equivalently,
\begin{equation}\label{e:multflagAnpartterm}
	[D(1,\lambda):D(2,\mu,r)]=\rho^{\bs a^{\lambda-\mu}}_{\bs\mu_0}\left(r-\frac{1}{2}(\lambda + \mu_1, \lambda-\mu)\right) \quad\text{for all}\quad \lambda,\mu\in P^+, \mu\le \lambda, \ r\in\mathbb Z.
\end{equation}

Setting
\begin{align*}
	\beta_k^{\mu,r} & = [D(1,\omega_i+k\theta ):D(2,\mu,r+k(1-\delta_{i,0}+k))]\\
	& \overset{\eqref{e:multflagAnpartterm}}{=} \rho^{\bs a^{\omega_i+k\theta-\mu}}_{\bs\mu_0}\left(r+k(1-\delta_{i,0}+k)-\frac{1}{2}(\omega_i+k\theta + \mu_1, \omega_i+k\theta-\mu)\right),
\end{align*} 
it follows from \cref{t:typeAn} that  
\begin{align*}	
	[V:V(\xi)] = \sum_{\mu\in\Gamma_{\xi}} \lim_{k\to\infty} \beta_k^{\mu,\xi} \quad\text{where}\quad \beta_k^{\mu,\xi}:=\beta_k^{\mu,r(\mu,\xi)}.
\end{align*}
To shorten notation, set
\begin{equation*}
	m_{i,\mu,r}= r+k(1-\delta_{i,0}+k)-\frac{1}{2}(\omega_i+k\theta + \mu_1, \omega_i+k\theta-\mu)
\end{equation*}
so 
\begin{equation*}
	\beta_k^{\mu,r} = \rho^{\bs a^{\omega_i+k\theta-\mu}}_{\bs\mu_0}(m_{i,\mu,r}).
\end{equation*}
Note
\begin{align*}
	m_{i,\mu,r} & = r+(1-\delta_{i,0})k +k^2 -\frac{1}{2}\left((\omega_i + \mu_1, \omega_i-\mu) + (\omega_i+\mu_1,k\theta) + (k\theta,\omega_i-\mu) + (k\theta,k\theta)\right)\\
	& = r+(1-\delta_{i,0})k +k^2 -\frac{1}{2}\left((\omega_i + \mu_1, \omega_i-\mu) + k(1-\delta_{i,0}+|\mu_1|) + k(1-\delta_{i,0}-|\mu|) + k^2\right)\\
	& = r -\frac{1}{2}\left((\omega_i + \mu_1, \omega_i-\mu) + k(|\mu_1|-|\mu|)\right)\\
	& = r+|\mu_0|k  -\frac{1}{2}(\omega_i + \mu_1, \omega_i-\mu).
\end{align*}
Using that $\mu_1=\mu-2\mu_0$, the last expression becomes
\begin{align*}
	m_{i,\mu,r} & = |\mu_0|k + (\mu_0,\omega_i-\mu) + r -\frac{1}{2}((\omega_i+\mu,\omega_i-\mu))\\
	& = |\mu_0|k -\langle \bs\mu_0, \bs a^{\mu-\omega_i}\rangle + r -\frac{(\omega_i,\omega_i) - (\mu,\mu)}{2}. 
\end{align*}
Then, by \eqref{e:socdeg'},
\begin{equation*}
	m_{i,\mu,r(\mu,\xi)} = |\mu_0|k -\langle \bs\mu_0, \bs a^{\mu-\omega_i}\rangle - f_{i,\xi}(\mu) \ \ \text{with}\ \ f_{i,\xi}(\mu) = \frac{ 2(\omega_i,\omega_i) - (\bar\xi,\bar\xi) -4\xi(d)-(\mu,\mu)}{4}. 
\end{equation*}
Also, one easily checks
\begin{equation*}
	\bs a ^{-\eta}= -\bs a^\eta \quad\text{and}\quad \bs a^{k\theta+\eta} = k+\bs a^\eta \quad\text{for all}\quad k\in\mathbb Z, \eta\in Q
\end{equation*}
and, hence,
\begin{equation*}
	\beta_k^{\mu,\xi} = \rho^{k-\bs a^{\mu-\omega_i}}_{\bs\mu_0}\left(  |\mu_0|k -\langle \bs\mu_0, \bs a^{\mu-\omega_i}\rangle - f_{i,\xi}(\mu)\right).
\end{equation*}
It then follows from \cref{l:mpart=0}(a) that
\begin{equation}
	f_{i,\xi}(\mu)<0 \quad\Rightarrow\quad \beta_k^{\mu,\xi} = 0.
\end{equation}
Thus, henceforth, assume $f_{i,\xi}(\mu)\ge 0$. An application of \cref{l:mpart=0}(b) then gives
\begin{equation*}
	\lim_{k\to\infty} \rho^{k-\bs a^{\mu-\omega_i}}_{\bs\mu_0}\left(  |\mu_0|k -\langle \bs\mu_0, \bs a^{\mu-\omega_i}\rangle - f_{i,\xi}(\mu)\right) = \rho_{\bs \mu_0}(f_{i,\xi}(\mu)),
\end{equation*}
thus completing the proof.

\bibliographystyle{amsplain}

\begin{thebibliography}{10}
\bibitem{har}
S. Al Harbat, {\em Canonical reduced expression for elements of affine Coxeter groups Part I -- Type $\tilde A_n$}, \url{https://doi.org/10.48550/arXiv.2105.07417}

\bibitem{andrews}
G. Andrews, The theory of partitions, Cambridge University Press (1998). 

\bibitem{bcsv:dfcpmt}
R. Biswal, V. Chari, L. Schneider, S. Viswanatha, {\em Demazure flags, Chebyshev polynomials, partial and mock theta functions}, Journal of Combinatorial Theory, Series A {\bf 140} (2016), 38--75, \url{https://doi.org/10.1016/j.jcta.2015.12.003}.

\bibitem{bcsw:mdpl2}
R. Biswal, V. Chari,  P. Shereen, and J. Wand, {\em Macdonald polynomials and level two Demazure modules for affine $\lie{sl}_{n+1}$}, {J. Alg. {\bf 575} (2021), Pages 159--191}, \url{https://doi.org/10.1016/j.jalgebra.2021.01.036}. 

\bibitem{bk}
R. Biswal and D. Kus, {\em A combinatorial formula for graded multiplicities in excellent filtrations}, {Transformation Groups 26, 81–114 (2021)}, \url{https://doi.org/10.1007/s00031-020-09574-4}.

\bibitem{bur}
T. Buger, Characterization of certain elements in the orbits of the affine Weyl group action
on the affine weight lattice for type $A^{(1)}_2$ via alcove walks. M.Sc. Thesis — University of
North Carolina Wilmington, 2017.


\bibitem{cssw:demflag}
V. Chari, L. Schneider, P. Shereen, and J. Wand, {\em Modules with Demazure flags and character formulae}, {SIGMA {\bf 10} (2014), 032, 16 pages}, \url{http://dx.doi.org/10.3842/SIGMA.2014.032}.

\bibitem{est:msc}
L. Estivalez, Combinatorial aspects of affine Weyl groups and orbits of dominant weights, MSc dissertation, Unicamp (2022). \url{https://hdl.handle.net/20.500.12733/4120}

\bibitem{humlie}
J.~Humphreys, Introduction to Lie algebras and Representation theory, Springer-Verlag (1972).


\bibitem{jamo:exf}
D. Jakeli\'c and A. Moura, {\em Limits of multiplicities in excellent filtrations and tensor product decompositions for affine Kac-Moody algebras}, {Algebr Represent Theor 21, 239–258 (2018)}. \url{https://doi.org/10.1007/s10468-017-9712-1}

\bibitem{jmmo}
M. Jimbo, K. Misra, T. Miwa, M. Okado, {\em Combinatorics of representations of $U_q(\hlie{sl}(n))$ at $q=0$}, {Commun.Math. Phys. 136, 543–566 (1991)}. \url{https://doi.org/10.1007/BF02099073}

\bibitem{kumar}
S. Kumar, {Kac-Moody groups, their flag varieties and representation theory}, Birkh\"auser (2002).

\bibitem{kura} D. Kus and V. Rappel, {\em Pieri formulas, higher level Demazure crystals and numerical multiplicities of excellent filtrations}, {J. Comb. Algebra (2025)}. \url{https://doi.org/10.4171/jca/108}

\bibitem{mm}
K. Misra, T. Miwa, {\em Crystal base for the basic representation of $U_q(\hlie{sl}(n))$}, {Commun.Math. Phys. 134, 79–88 (1990)}. \url{https://doi.org/10.1007/BF02102090}

\bibitem{mw:tpgen}
K. Misra and E. Wilson, {\em Tensor product decomposition of $\hlie{sl}(n)$-modules and identities}, {Contemporary Mathematics {\bf 627} (2014), 131--144}. \url{http://dx.doi.org/10.1090/conm/627/12538}


\end{thebibliography}

\end{document}